\documentclass[review,onefignum,onetabnum]{siamonline190516}
\usepackage{color}
\usepackage{amssymb}
\usepackage{tikz}
\newcommand{\be}{\begin{eqnarray}}
\newcommand{\ben}{\begin{eqnarray*}}
\newcommand{\en}{\end{eqnarray}}
\newcommand{\enn}{\end{eqnarray*}}
\newcommand{\R}{\mathbb{R}}
\newcommand{\Z}{\mathbb{Z}}

\newcommand{\C}{\mathbb{C}}

\definecolor{rot}{rgb}{1.000,0.000,0.000}
\definecolor{blue}{rgb}{0.000,0.000,1.000}

\usepackage{lipsum}
\usepackage{mathrsfs}
\usepackage{amsfonts}
\usepackage{graphicx}
\usepackage{epstopdf}
\usepackage{algorithmic}
\ifpdf
  \DeclareGraphicsExtensions{.eps,.pdf,.png,.jpg}
\else
  \DeclareGraphicsExtensions{.eps}
\fi

\usepackage{enumitem}
\setlist[enumerate]{leftmargin=.5in}
\setlist[itemize]{leftmargin=.5in}

\newsiamremark{remark}{Remark}
\newsiamremark{hypothesis}{Hypothesis}
\crefname{hypothesis}{Hypothesis}{Hypotheses}
\newsiamthm{claim}{Claim}

\title{Uniqueness in determining rectangular grating profiles with a single incoming wave (Part II): TM polarization case \thanks{Submitted to the editors DATE.
}
}

\author{Jianli Xiang\thanks{Three Gorges Mathematical Research Center, College of Science, China Three Gorges University, Yichang 443002, People's Republic of China
  (\email{xiangjianli@ctgu.edu.cn}).}
\and Guanghui Hu \thanks{Corresponding author: School of Mathematical Sciences and LPMC, Nankai University, Tianjin 300071, People's Republic of China (\email{ghhu@nankai.edu.cn}).}
}

\usepackage{amsopn}

\ifpdf
\hypersetup{
  pdftitle={uniqueness in determining rectangular periodic structures},
  pdfauthor={Jianli Xiang}
}
\fi

\nolinenumbers
\begin{document}

\maketitle

\begin{abstract}
  This paper is concerned with an inverse transmission problem for recovering the shape of a penetrable rectangular grating sitting on a perfectly conducting plate. We consider a general transmission problem with the coefficient $\lambda\neq 1$ which covers the TM polarization case. It is proved that a rectangular grating profile can be uniquely determined by the near-field observation data incited by a single plane wave and measured on a line segment above the grating. In comparision with the TE case ($\lambda=1$), the wave field cannot lie in $H^2$ around each corner point, bringing essential difficulties in proving uniqueness  with one plane wave. Our approach relies on singularity analysis for Helmholtz transmission problems in a right-corner domain and also provides an alternative idea for treating the TE transmission conditions which were considered in the authors' previous work [Inverse Problem, 39 (2023): 055004.]
  \end{abstract}



\section{Introduction and main result}\label{sec2}

Consider the time-harmonic electromagnetic scattering of a plane wave from a penetrable rectangular grating which remains invariant along one surface direction $x_3$. The diffractive grating is supposed to sit on the perfectly conducting substrate $x_2<0$. In TE and TM polarization cases, the wave scattering can be modeled by a transmission problem for the Helmholtz equation over the $ox_1x_2$-plane with a boundary condition on $x_2=0$ and an appropriate radiation condition as $x_2\rightarrow\infty$.
In this paper the medium above the grating profile is supposed to be isotropic and homogeneous.
For rectangular gratings,  the cross-section $\Lambda$ of the grating surface in the $ox_1x_2$-plane consists of line segments that are perpendicular to either the $x_1$ or  $x_2$-axis. More precisely, we define a set $\mathcal{A}$ of the so-called rectangular grating profiles
 by (see Figure \ref{fig0})
\begin{align*}
\mathcal{A}=\big\{\Lambda~|~&\Lambda\mbox{ is a non-self-intersecting  curve in $\mathbb{R}_{+}^{2}$ which is $2\pi$-periodic  in $x_1$, } \\ &\mbox{$\Lambda$ is piecewise linear and any linear part is parallel to the $x_{1}$- or $x_{2}$-axis} \big\}.
\end{align*}
Note that $\Lambda\in \mathcal{A}$ cannot contain any crack, for instance, a line segment intersecting the other part of $\Lambda$  at one ending point. The rectangular gratings defined above include the class of binary gratings, whose grooves have the same height. Denote by $\Omega_{\Lambda}^{+}$ the unbounded periodic domain above $\Lambda$, that is, the component of $\mathbb{R}^{2}_+$ separated by $\Lambda$ which is connected to $x_{2}=+\infty$. Let $\Omega_{\Lambda}^{-}$ be the periodic domain below $\Lambda$ but above the substrate $x_2=0$. Let $\nu=(\nu_1, \nu_2)\in \mathbb{S}:=\{x\in \mathbb{R}^2: |x|=1\}$ be the normal direction at $\Lambda$ pointing into $\Omega^+_\Lambda$. 
Suppose that a plane wave in the $(x_{1},x_{2})$-plane given by
\begin{equation*}
u^{i}(x_{1},x_{2})=e^{i\alpha x_{1}-i\beta x_{2}}, \quad \alpha=k_1\sin\theta, \quad \beta=k_1\cos\theta
\end{equation*}
with some incident angle $\theta\in(-\pi/2, \pi/2)$ and wave number $k_1>0$, is incident upon the grating $\Lambda$ from the top. Consider a general transmission problem  for finding the total field $u=u(x_{1},x_{2})$ such that
\begin{equation} \label{a}
\left\{\begin{array}{lll}
 \Delta u+k_1^{2}u=0, & \quad \mbox{in} \quad \Omega_{\Lambda}^{+}, \vspace{0.2cm} \\
 \Delta u+k_2^2 u=0,& \quad \mbox{in} \quad \Omega_{\Lambda}^{-}, \vspace{0.2cm} \\
 u^+=u^-, \quad \partial_\nu^+u=\lambda\,\partial_\nu^-u,  & \quad \mbox{on} \quad \Lambda, \vspace{0.2cm} \\
 u=u^{i}+u^{s},& \quad \mbox{in} \quad \Omega_{\Lambda}^{+},\vspace{0.2cm}\\
 \partial_\nu u=0,& \quad \mbox{on} \quad \Gamma_0,
 \end{array}\right.
\end{equation}
with the following radiation condition as $x_{2}\rightarrow+\infty$:
\begin{equation} \label{rad1}
u^{s}(x):=u-u^i=\sum_{n\in\mathbb{Z}}A_{n}~e^{i\alpha_{n}x_{1}+i\beta_{n}x_{2}}\qquad \mbox{in}~~x_{2}>\Lambda^{+}:=\max_{(x_{1},x_{2})\in\Lambda}x_{2}.
\end{equation}
In \eqref{a}, we have $k_j>0$ for $j=1,2$, $k_1\neq k_2$, $\lambda>0, \lambda\neq 1$,  $\alpha_{n}:=n+\alpha$ and
\ben
\beta_{n}:=\left\{\begin{array}{lll}
\sqrt{k_1^{2}-\alpha_{n}^{2}}\quad &\mbox{if}~~|\alpha_{n}|\leq k_1,\vspace{0.3cm}\\
i\sqrt{\alpha_{n}^{2}-k_1^{2}} \quad &\mbox{if}~~|\alpha_{n}|> k_1.
\end{array}\right.
\enn
The notation $(\cdot)^\pm$ stand for the limits of $u$ and $\partial_\nu u$ on $\Lambda$ obtained from above $(+)$ or below $(-)$ and $\Gamma_{H}=\{(x_{1},H): 0<x_{1}<2\pi\} $ for $H\in \R$. Note that the TM polarization case corresponds to the special  case that $\lambda=(k_1/k_2)^2$.
The expansion in (\ref{rad1}) is the well-known Rayleigh expansion (see e.g. \cite{Hettlich1997,Ray1907,Petit1980}), $A_{n}\in\mathbb{C}$ are called Rayleigh coefficients. 
The series \eqref{rad1} together with their derivatives are uniform convergent in any compact set in $x_2>\Lambda^+$, because
$u\in H_\alpha^1(S_H)$ (see below for the definition) and
the scattered fields consist of infinitely many surface waves which exponentially decay as $x_2\rightarrow+\infty$.
\begin{figure} \label{fig0}
  \centering
  \begin{tikzpicture}[scale=0.75]
  \draw[gray, thick,->] (0,-2) -- (0,4); \draw (0.1,4.2) node{$x_2$};
  \draw[gray, thick,->] (-2,0) -- (8.6,0); \draw (8.9,0) node{$x_1$};
  \filldraw[black] (0,0) circle (0.5pt);\draw (-0.2,0.2) node{$O$};
  \draw[gray, thick] (0,3) -- (1,3)-- (1,0.5) -- (2,0.5) -- (2,2.5) --(3,2.5) --(3,1.3) --(4,1.3) --(4,3)--(5,3)--(5,0.5)--(6,0.5)--(6,2.5)--(7,2.5)--(7,1.3)--(8,1.3)--(8,3);
  \draw[gray, thick, dashed] (4,-0.5) -- (4,3.5); \draw (4,-0.8) node{$x_{1}=2\pi$};
  \draw[gray, thick, dashed] (8,-0.5) -- (8,3.5); \draw (8,-0.8) node{$x_{1}=4\pi$};
  \draw[gray, thick, dashed] (0,3.3) -- (4,3.3); \draw (2,3.3) node{$\Gamma_{H}$};
  \draw (2.5,2.65) node{$\Lambda$}; \draw (3,3.8) node{$\Omega_{\Lambda}^{+}$};
  \draw (3,0.5) node{$\Omega_{\Lambda}^{-}$}; \draw (2.5,0) node{$\Gamma_{0}$};
  \end{tikzpicture}
  \caption{Rectangular periodic structures.}
\end{figure}
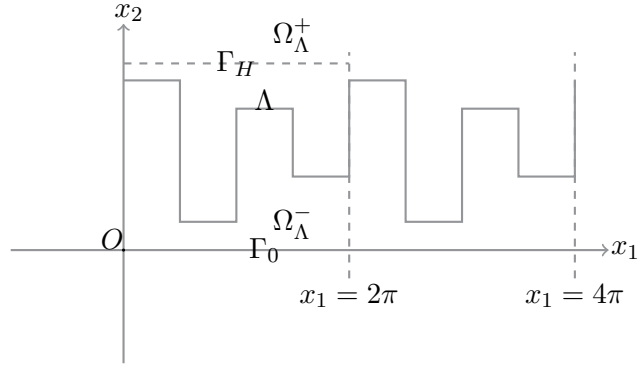
We will look for weak solutions to \eqref{a}--\eqref{rad1} in the $\alpha$-quasiperiodic Sobolev space
\begin{equation*}
H^1_\alpha(S_H):=\big\{u\in H^1_{{\rm loc}}(S_H),~e^{-i\alpha x_1}u\;\mbox{is $2\pi$-periodic in $x_1$}\big\}, 
\end{equation*}
with $S_H:=\{x\in \mathbb{R}^2: 0<x_2<H\}$ for any $H>\Lambda^+$.
Note that, since we are interested in quasi-periodic solutions, the notations $\Omega_{\Lambda}^{\pm}, \Lambda, S_{H}$ and $\Gamma_{H}$ always denote the corresponding sets in one periodicity cell $0<x_{1}<2\pi$. Uniqueness, existence and regularity results on solutions to the forward scattering problem  will be summarized as follows.
\begin{proposition}\label{prop}
\begin{description}
\item[(i)] There exists at least one solution $u\in H_\alpha^1(S_H)$ to the forward scattering problem (\ref{a})--(\ref{rad1}), where $H>\Lambda^+$ is arbitrary. Moreover, uniqueness holds true if $k_1^2\geq \lambda k_2^2$.
\item[(ii)] Let $u\in H_\alpha^1(S_H)$ be a solution to the forward scattering problem (\ref{a})--(\ref{rad1}) corresponding to some rectangular grating $\Lambda\in\mathcal{A}$. Then we have $u\in H_\alpha^{1+s}(S_H)\cap H^2_\alpha(S_H^\pm)$ for any $s\in [0, 1/2)$, where $S_H^\pm:=S_H\cap \Omega_\Lambda^\pm$. Moreover, $u$ is real-analytic on $\overline{S_H^+}$ and $\overline{S_H^-}$ except at the finite number of corner points of $\Lambda$.
\end{description}
\end{proposition}

Uniqueness and existence of the above transmission problem have been sufficiently investigated in the literature by applying the Dirichlet-to-Neumann map; see e.g., \cite{Bao2001, BL, BS, ES98} in periodic structures. In particular, the uniqueness proof for rectangular gratings with the condition $k_1^2\geq \lambda k_2^2$ follows directly from the authors' previous paper \cite[Appendix]{XH}. If $k_1^2\geq \lambda k_2^2$ does not hold, guided bloch waves might exist and additional constraint should be imposed on the total field to ensure uniqueness; see the recent publication \cite{HK22} for a sharp radiation condition  derived from the limiting absorption principle under the Dirichlet boundary condition. The second assertion, which states smoothness of the solution around a corner point and  up to a flat interface, follows from standard elliptic regularity result for interface problems in a right-corner domain; see e.g., in \cite{ES98,K1967,KMR,MNP,Pe2001}. We refer to the Appendix of this paper for the proof of Proposition \ref{prop}.

Now we formulate the inverse problem with a single measurement data above the grating.
\begin{description}
  \item{(IP):} Let $H>\Lambda^{+}$ be a fixed constant and suppose $u=u(x_{1},x_{2})$ is a solution to the direct problem (\ref{a})--(\ref{rad1}). Given the transmission coefficient $\lambda>0$ $(\neq1)$ and the wavenumbers $k_1$ and $k_2$, determine the periodic interface $\Lambda\in\mathcal{A}$ from knowledge of the near-field data $u(x_{1},H)$ for all $0<x_{1}<2\pi$.
\end{description}
The main uniqueness result of this paper is stated as follows.
\begin{theorem} \label{Main}
Let $u_{1}$ and $u_{2}$ be solutions to the direct diffraction problem (\ref{a})--(\ref{rad1}) corresponding to $(\Lambda_{1},k_{1},k_{2},\lambda)$ and $(\Lambda_{2},k_{1},k_{2},\lambda)$, respectively. If
\begin{equation}
u_{1}(x_{1},H)=u_{2}(x_{1},H)\quad {\rm for~all~}x_{1}\in(0,2\pi),
\end{equation}
where $H>\max\{\Lambda_{1}^{+},\Lambda^{+}_{2}\}$ is a fixed constant, then $\Lambda_{1}=\Lambda_{2}$.
\end{theorem}
It is well-known that a general grating profile cannot be uniquely determined by one plane wave in a lossless media. In the literature there are uniqueness results using many incoming waves of different kinds, for instance, quasiperiodic waves with the same phase-shift \cite{Kirsch1994}, fixed-direction multifrequency plane waves \cite{Hettlich1997} and fixed-frequency multi-direction plane waves \cite{XH23}. Binary gratings have very important applications in industry, because they can be easily fabricated \cite{SK97, Turunen1997}. The inverse problem of identifying parameters of binary gratings plays a major role in quality control and optimal design of diffractive elements with prescribed far field patterns \cite{Bao2001,D93,ES98}. In the authors' previous work \cite{XH}, a global uniqueness result in the TE polarization case (i.e., $\lambda=1$) was verified. The approach of \cite{XH} was based on the singularity analysis of an overdetermined Cauchy problem for an inhomogeneous Laplacian equation in a corner domain. If $\lambda\neq 1$, the wave field cannot lie in $H^2$ around each corner point. This weaker smoothness gives rise to essential difficulties in carrying out approach of \cite{XH} to the transmission conditions with $\lambda\neq 1$. The aim of this paper is to develop a different approach for proving Theorem \ref{Main}. Numerically, optimization-based iterative schemes are usually utilized for solving the inverse problem. One may conclude from Theorem \ref{Main} that the global minimizer of the object functional within the class of rectangular gratings is unique. The proof of Theorem \ref{Main} also implies that wave fields must be singular (that is, non-analytic) at the corner point.

\section{Preliminary lemmas}
The singularity analysis seems natural for justifying uniqueness  to inverse scattering from penetrable scatterers whose boundary contains corner points; see e.g. \cite{Elschner2018,ElschnerJ2015,XH} where the TE transmission conditions (i.e., $\lambda=1$) were considered. 
As will be seen later, the TM case appears quite different from the TE case.
In this section, we prepare several lemmas for the proof of Theorem \ref{Main}. They are mostly motivated by the papers \cite{Elschner2018,ElschnerJ2015,XH},
but are interesting on their own right. Throughout the whole paper, we let $(r,\theta)$ be the polar coordinates of $x=(x_{1},x_{2})$ in $\mathbb{R}^{2}$, and let $B_{R}$ denote the disk centered at origin with radius $R>0$.
The corner domains $\Omega_{\ell}$ and the line segments $\Pi_{\ell}$ ($\ell=1,2$) are defined as (see Figure \ref{f1}):
\ben
&&\Omega_{1}:=\{(r,\theta):0<r<R,~0<\theta<3\pi/2\}, \quad ~~\Pi_{1}:=\{(r,0):0\leq r\leq R\},\\
&&\Omega_{2}:=\{(r,\theta):0<r<R,~-\pi/2<\theta<0\},   
\quad\Pi_{2}:=\{(r,3\pi/2):0\leq r\leq R\}.  
\enn
\begin{figure}[h] \label{f1}
  \centering
  \begin{tikzpicture}[scale=0.75]  \filldraw[gray,opacity=0.5] (0,0)--(0,-2.5) arc (270:360:2.5);
  \filldraw[gray!30,opacity=0.5] (0,0)--(2.5,0) arc (0:270:2.5);
  \draw[gray, thick,->] (0,-3) -- (0,3.5); \draw (0.1,3.7) node{$x_2$};
  \draw[gray, thick,->] (-3,0) -- (3.5,0); \draw (3.8,0) node{$x_1$};
  \filldraw[black] (0,0) circle (0.5pt); \draw (-0.2,0.2) node{$O$}; \filldraw[black] (0,0) circle (0.8pt);
  \draw[gray,thick] (0,0) circle (2.5cm); \draw (2.2,2.2) node{$B_{R}$};
  \draw (1,-1) node{$\Omega_{2}$}; \draw (-1,1) node{$\Omega_{1}$};
  \draw[red, line width=1] (0,0) -- (2.5,0); \draw (1.25,0.2) node{$\Pi_{1}$};
  \draw[red, line width=1] (0,0) -- (0,-2.5); \draw (-0.25,-1.25) node{$\Pi_{2}$};
  \end{tikzpicture}
  \caption{Illustration of two domains $\Omega_{\ell}$ and two line segments  $\Pi_{\ell}$ ($\ell=1,2$).}
\end{figure}
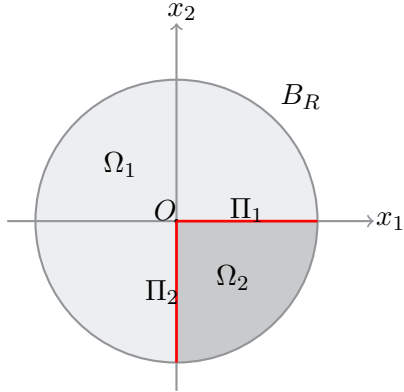

\begin{lemma} \label{Th1}
Let $q_{1}$ and $q_{2}$ be two constants in $B_{R}$ and let $\lambda$ be a positive constant. Suppose that $u_{1}$ and $u_{2}$ satisfy the Helmholtz equations
\begin{equation*}
\Delta u_{\ell}+q_{\ell}u_{\ell}=0\quad {\rm in}~B_{R},\quad \ell=1,2
\end{equation*}
subject to the transmission conditions
\begin{equation} \label{trans}
u_{1}=u_{2}, \quad \frac{\partial u_{1}}{\partial\nu}=\lambda\frac{\partial u_{2}}{\partial\nu} \quad {\rm on}~~\Pi_{1}\cup\Pi_{2}.
\end{equation}
If $q_{1}\neq q_{2}$ and $\lambda\neq1$, then $u_{1}=u_{2}\equiv0$ in $B_{R}$.
\end{lemma}
\begin{proof}
Recalling the Taylor expansion of analytic solutions of the Helmholtz equation (see \cite{ElschnerJ2015,EHY}), we have
\begin{equation*}
u_{\ell}(r,\theta)=\sum_{n,m\in \mathbb{N}:n+2m\geq0}r^{n+2m} \Big(a^{(\ell)}_{n,m}\cos(n\theta)+b^{(\ell)}_{n,m}\sin(n\theta)\Big),\quad {\rm for}~~0\leq r<R,
\end{equation*}
where the coefficients $a_{n,m}^{(\ell)}$ and $b_{n,m}^{(\ell)}$ fulfill the recurrence relations
\begin{equation} \label{recur1}
a^{(\ell)}_{n,m+1}=\frac{-q_{\ell}}{4(m+1)(n+m+1)}a^{(\ell)}_{n,m},~ b^{(\ell)}_{n,m+1}=\frac{-q_{\ell}}{4(m+1)(n+m+1)}b^{(\ell)}_{n,m},~\forall~n,m \in \mathbb{N}.
\end{equation}
The transmission conditions in (\ref{trans}) are equivalent to the four relations:
\begin{equation*}
\sum\limits_{n,m\in\mathbb{N}}^{n+2m=l}a^{(1)}_{n,m}=\sum\limits_{n,m\in\mathbb{N}}^{n+2m=l}a^{(2)}_{n,m},\quad \quad
\sum\limits_{n,m\in\mathbb{N}}^{n+2m=l}nb^{(1)}_{n,m}=\lambda\sum\limits_{n,m\in\mathbb{N}}^{n+2m=l}nb^{(2)}_{n,m},
\end{equation*}
\begin{align*}
\sum\limits_{n,m\in \mathbb{N}}^{n+2m=l}\!\!\!\big[a^{(1)}_{n,m}\cos(n\pi/2)-b^{(1)}_{n,m}\sin(n\pi/2)\big]
=\sum\limits_{n,m\in\mathbb{N}}^{n+2m=l}\!\!\!\big[a^{(2)}_{n,m}\cos(n\pi/2)-b^{(2)}_{n,m}\sin(n\pi/2)\big], \\
\sum\limits_{n,m\in\mathbb{N}}^{n+2m=l}n\big[a^{(1)}_{n,m}\sin(n\pi/2) +b^{(1)}_{n,m}\cos(n\pi/2)\big]
=\lambda\sum\limits_{n,m\in\mathbb{N}}^{n+2m=l}n\big[a^{(2)}_{n,m}\sin(n\pi/2) +b^{(2)}_{n,m}\cos(n\pi/2)\big].
\end{align*}

{\bf Case One}: $n=2\mathrm{k}+1$ for some $\mathrm{k}\in\mathbb{N}$. 
In this case the transmission conditions can be simplified to be
\begin{align}  \label{re1}
\left\{\begin{array}{lll}
 &\sum\limits_{2\mathrm{k}+1+2m=l}a^{(1)}_{2\mathrm{k}+1,m} =\sum\limits_{2\mathrm{k}+1+2m=l}a^{(2)}_{2\mathrm{k}+1,m},\vspace{0.2cm} \\
 &\sum\limits_{2\mathrm{k}+1+2m=l}(2\mathrm{k}+1)(-1)^{\mathrm{k}}\,a^{(1)}_{2\mathrm{k}+1,m}
 =\lambda\sum\limits_{2\mathrm{k}+1+2m=l}(2\mathrm{k}+1)(-1)^{\mathrm{k}}\,a^{(2)}_{2\mathrm{k}+1,m},
\end{array}\right.
\end{align}
\begin{align}  \label{re2}
\left\{\begin{array}{lll}
 &\sum\limits_{2\mathrm{k}+1+2m=l}(2\mathrm{k}+1)b^{(1)}_{2\mathrm{k}+1,m} =\lambda\sum\limits_{2\mathrm{k}+1+2m=l}(2\mathrm{k}+1)b^{(2)}_{2\mathrm{k}+1,m}, \vspace{0.2cm} \\
 &\sum\limits_{2\mathrm{k}+1+2m=l}(-1)^{\mathrm{k}}\,b^{(1)}_{2\mathrm{k}+1,m}
 =\sum\limits_{2\mathrm{k}+1+2m=l}(-1)^{\mathrm{k}}\,b^{(2)}_{2\mathrm{k}+1,m}.
\end{array}\right.
\end{align}
It suffices to show $a^{(\ell)}_{2\mathrm{k}+1,m}=b^{(\ell)}_{2\mathrm{k}+1,m}=0$ for all $\mathrm{k},m\in \mathbb{N}$, $\ell=1,2$.

We first consider the case: $l=2\mathrm{k}+1+2m=1$, that is $\mathrm{k}=0$, $m=0$. From \eqref{re1} and \eqref{re2} we deduce that
\begin{equation*}
a^{(1)}_{1,0}=a^{(2)}_{1,0}, \quad  a^{(1)}_{1,0}=\lambda a^{(2)}_{1,0}; \quad\quad  b^{(1)}_{1,0}=\lambda b^{(2)}_{1,0}, \quad b^{(1)}_{1,0}=b^{(2)}_{1,0}.
\end{equation*}
Since $\lambda\neq1$, we obtain $a^{(1)}_{1,0}=a^{(2)}_{1,0}=b^{(1)}_{1,0}=b^{(2)}_{1,0}=0$. By the recurrence relation (\ref{recur1}), we have $a^{(\ell)}_{1,m} =b^{(\ell)}_{1,m}=0$ for all $m\in \mathbb{N}$, $\ell=1,2$.

We carry out the proof by induction. Supposing for some $M\in \mathbb{N}$ that
\begin{equation}  \label{induc1}
a^{(1)}_{2\mathrm{k}+1,m}=a^{(2)}_{2\mathrm{k}+1,m}=0,\quad b^{(1)}_{2\mathrm{k}+1,m}=b^{(2)}_{2\mathrm{k}+1,m}=0, \quad {\rm for}~~\mathrm{k}\leq M,\quad \mathrm{k},m\in\mathbb{N}.
\end{equation}
We need to prove the above relations in (\ref{induc1}) with $M$ replaced by $M+1$. For this purpose, it is sufficient to verify
\begin{equation*}
a^{(1)}_{2M+3,0}=a^{(2)}_{2M+3,0}=0, \quad b^{(1)}_{2M+3,0}=b^{(2)}_{2M+3,0}=0.
\end{equation*}
Setting $l=2\mathrm{k}+1+2m=2M+3$ in (\ref{re1}) and (\ref{re2}) and using the relations in (\ref{induc1}), we obtain
\begin{equation*}
a^{(1)}_{2M+3,0}=a^{(2)}_{2M+3,0},\quad  a^{(1)}_{2M+3,0}=\lambda a^{(2)}_{2M+3,0};\quad\quad b^{(1)}_{2M+3,0}=\lambda b^{(2)}_{2M+3,0}, \quad
b^{(1)}_{2M+3,0}=b^{(2)}_{2M+3,0}.
\end{equation*}
Again using $\lambda\neq1$ yields $a^{(1)}_{2M+3,0}=a^{(2)}_{2M+3,0}=b^{(1)}_{2M+3,0}=b^{(2)}_{2M+3,0}=0$. Consequently, we achieve that $a^{(\ell)}_{2\mathrm{k}+1,m}=b^{(\ell)}_{2\mathrm{k}+1,m}=0$ for all $\mathrm{k},m\in\mathbb{N}$, $\ell=1,2$.

{\bf Case Two}: $n=2\mathrm{k}$ for $\mathrm{k}\in\mathbb{N}$. 
It then follows from the transmission conditions that
\begin{equation}  \label{re3}
\sum\limits_{2\mathrm{k}+2m=l}a^{(1)}_{2\mathrm{k},m} =\sum\limits_{2\mathrm{k}+2m=l}a^{(2)}_{2\mathrm{k},m},\quad\quad
 \sum\limits_{2\mathrm{k}+2m=l}(-1)^{\mathrm{k}}\,a^{(1)}_{2\mathrm{k},m}
=\sum\limits_{2\mathrm{k}+2m=l}(-1)^{\mathrm{k}}\,a^{(2)}_{2\mathrm{k},m},
\end{equation}
\begin{equation}  \label{re4}
\sum\limits_{2\mathrm{k}+2m=l}\mathrm{k}\,b^{(1)}_{2\mathrm{k},m}
=\lambda\sum\limits_{2\mathrm{k}+2m=l}\mathrm{k}\,b^{(2)}_{2\mathrm{k},m}, \quad\quad
\sum\limits_{2\mathrm{k}+2m=l}(-1)^{\mathrm{k}}\,\mathrm{k}\,b^{(1)}_{2\mathrm{k},m}
=\lambda\sum\limits_{2\mathrm{k}+2m=l}(-1)^{\mathrm{k}}\,\mathrm{k}\,b^{(2)}_{2\mathrm{k},m}.
\end{equation}

Suppose $\tilde{l}:=\mathrm{k}+m=0$, that is $\mathrm{k}=0$, $m=0$. From the relation (\ref{re3}), we obtain $a^{(1)}_{0,0}=a^{(2)}_{0,0}$.
Then we set $\tilde{l}=\mathrm{k}+m=1$ in (\ref{re3}) and (\ref{re4}), that is $\mathrm{k}=1$, $m=0$ or $\mathrm{k}=0$, $m=1$. This gives the relations $b^{(1)}_{2,0}=\lambda b^{(2)}_{2,0}$ and
\begin{equation*}
a^{(1)}_{2,0}+a^{(1)}_{0,1}=a^{(2)}_{2,0}+a^{(2)}_{0,1}, \quad  -a^{(1)}_{2,0}+a^{(1)}_{0,1}=-a^{(2)}_{2,0}+a^{(2)}_{0,1},
\end{equation*}
which imply that $a^{(1)}_{0,1}=a^{(2)}_{0,1}$ and $a^{(1)}_{2,0}=a^{(2)}_{2,0}$. Since $a^{(1)}_{0,0}=a^{(2)}_{0,0}$, $a^{(1)}_{0,1}=a^{(2)}_{0,1}$, $a^{(\ell)}_{0,1}= -\frac{q_{\ell}}{4}a^{(\ell)}_{0,0}$ and $q_{1}\neq q_{2}$, we obtain that
\begin{equation*}
a^{(1)}_{0,m}=a^{(2)}_{0,m}=0, \quad \forall ~m\in\mathbb{N}.
\end{equation*}

Set $\tilde{l}=\mathrm{k}+m=2$ in (\ref{re3}) and (\ref{re4}), that is $\mathrm{k}=2$, $m=0$ or $\mathrm{k}=1$, $m=1$ or $\mathrm{k}=0$, $m=2$, we have
\begin{equation*}
\left\{\begin{array}{lll}
  a^{(1)}_{4,0}+a^{(1)}_{2,1}=a^{(2)}_{4,0}+a^{(2)}_{2,1}, \vspace{0.2cm} \\
  a^{(1)}_{4,0}-a^{(1)}_{2,1}=a^{(2)}_{4,0}-a^{(2)}_{2,1},
\end{array}\right. \quad\quad \quad
\left\{\begin{array}{lll}
  2b^{(1)}_{4,0}+b^{(1)}_{2,1}=\lambda\big(2b^{(2)}_{4,0}+b^{(2)}_{2,1}\big), \vspace{0.2cm} \\
  2b^{(1)}_{4,0}-b^{(1)}_{2,1}=\lambda\big(2b^{(2)}_{4,0}-b^{(2)}_{2,1}\big),
\end{array}\right.
\end{equation*}
which lead to that
\begin{equation*}
a^{(1)}_{4,0}=a^{(2)}_{4,0},\quad a^{(1)}_{2,1}=a^{(2)}_{2,1}; \quad\quad  b^{(1)}_{4,0}=\lambda b^{(2)}_{4,0},\quad b^{(1)}_{2,1}=\lambda b^{(2)}_{2,1}.
\end{equation*}
Since $a^{(1)}_{2,0}=a^{(2)}_{2,0}$, $a^{(1)}_{2,1}=a^{(2)}_{2,1}$, $a^{(\ell)}_{2,1}=-\frac{q_{\ell}}{12}a^{(\ell)}_{2,0}$ and $q_{2}\neq q_{1}$, we conclude that
\begin{equation*}
a^{(1)}_{2,m}=a^{(2)}_{2,m}=0, \quad \forall~m\in\mathbb{N}.
\end{equation*}
Since $b^{(1)}_{2,0}=\lambda b^{(2)}_{2,0}$, $b^{(1)}_{2,1}=\lambda b^{(2)}_{2,1}$ and $b^{(\ell)}_{2,1}=-\frac{q_{\ell}}{12}b^{(\ell)}_{2,0}$, we arrive at
\begin{equation*}
0=b^{(1)}_{2,1}-\lambda b^{(2)}_{2,1}=-\frac{q_{1}}{12}b^{(1)}_{2,0} +\lambda\frac{q_{2}}{12}b^{(2)}_{2,0}=\lambda\frac{q_{2}-q_{1}}{12}b^{(2)}_{2,0}.
\end{equation*}
That is $b^{(2)}_{2,0}=0$ for $q_{2}\neq q_{1}$, $\lambda\neq0$. By the recurrence relation (\ref{recur1}), we conclude
\begin{equation*}
b^{(1)}_{2,m}=b^{(2)}_{2,m}=0, \quad \forall~m\in\mathbb{N}.
\end{equation*}

We shall finish the proof by induction. Supposing for some $M\in \mathbb{N}$ that
\begin{equation} \label{induc2}
a^{(1)}_{2\mathrm{k}-2,m}=a^{(2)}_{2\mathrm{k}-2,m}=0,\quad a^{(1)}_{2M,0}=a^{(2)}_{2M,0}, \quad {\rm for}~~1\leq \mathrm{k}\leq M,\quad m\in\mathbb{N};
\end{equation}
\begin{equation}  \label{induc3}
b^{(1)}_{2\mathrm{k}-2,m}=b^{(2)}_{2\mathrm{k}-2,m}=0,\quad b^{(1)}_{2M,0}=\lambda b^{(2)}_{2M,0}, \quad {\rm for}~~1\leq \mathrm{k}\leq M,\quad m\in\mathbb{N}.
\end{equation}
We need to prove all relations in (\ref{induc2}) and (\ref{induc3}) with $M$ replaced by $M+1$. For this purpose, it is sufficient to verify
\begin{equation*}
a^{(1)}_{2M,0}=a^{(2)}_{2M,0}=0, \quad a^{(1)}_{2(M+1),0}=a^{(2)}_{2(M+1),0};\quad\quad
b^{(1)}_{2M,0}=b^{(2)}_{2M,0}=0, \quad b^{(1)}_{2M+2,0}=\lambda b^{(2)}_{2M+2,0}.
\end{equation*}
Setting $\tilde{l}=\mathrm{k}+m=M+1$ in (\ref{re3}) and using (\ref{induc2}), we obtain
\begin{equation*}
a^{(1)}_{2(M+1),0}+a^{(1)}_{2M,1}=a^{(2)}_{2(M+1),0}+a^{(2)}_{2M,1}, \quad\quad a^{(1)}_{2(M+1),0}-a^{(1)}_{2M,1}=a^{(2)}_{2(M+1),0}-a^{(2)}_{2M,1}.
\end{equation*}
That is, $a^{(1)}_{2(M+1),0}=a^{(2)}_{2(M+1),0}$ and $a^{(1)}_{2M,1}=a^{(2)}_{2M,1}$. Since $a^{(1)}_{2M,1}=a^{(2)}_{2M,1}$, $a^{(1)}_{2M,0}=a^{(2)}_{2M,0}$, $a^{(\ell)}_{2M,1}=\frac{-q_{\ell}}{4(2M+1)}a^{(\ell)}_{2M,0}$ and $q_{1}\neq q_{2}$, it follows that
$a^{(1)}_{2M,0}=a^{(2)}_{2M,0}=0$. Similarly, setting $\tilde{l}=\mathrm{k}+m=M+1$ in (\ref{re4}) and using (\ref{induc3}) will lead to
$b^{(1)}_{2(M+1),0}=\lambda b^{(2)}_{2(M+1),0}$ and 
$b^{(1)}_{2M,0}=b^{(2)}_{2M,0}=0$. 
\end{proof}

In our uniqueness proof, we need a weak version of Lemma \ref{Th1}, which is stated below.
\begin{lemma} \label{Th2}
Suppose $\rho_{1}(r,\theta)\equiv0$ in $\Omega_{1}$ and  $\rho_{1}(r,\theta)\equiv \rho\in \C, \rho\neq 0$ in $\Omega_{2}$. Let $v_{1}$, $v_{2}$ be solutions to
\begin{align*}
\Delta v_{1}+k^{2}(1+\rho_{1})v_{1}=0,\quad
\Delta v_{2}+k^{2}v_{2}=0\quad {\rm in}~B_{R},
\end{align*}
subject to the transmission conditions \eqref{trans}.
Then $v_{1}=v_{2}\equiv0$ in $B_{R}$.
\end{lemma}
\begin{proof}
Set $q_{1}:=k^{2}(1+\rho_{1})$ in $\Omega_{2}$. Since the Cauchy data of $v_{2}$ are analytic on $\Pi_{1}\cup\Pi_{2}$, the Cauchy data of $v_{1}$ are also analytic there by the transmission boundary conditions. Since $v_{1}$ is analytic in $\Omega_2$, by the Cauchy-Kowalewski theorem in a piecewise analytic domain (see \cite[Lemma 2.1]{LHY}), the function $v_1$ can be analytically extended from $\Omega_2$ to a full neighboring area of the corner as a solution of the Helmholtz equation $\Delta w_{1}+q_{1}w_{1}=0$, where $w_1$ denotes the extended solution. Now applying Lemma \ref{Th1} to $w_1$ and $v_2$ gives $w_{1}=v_{2}\equiv0$ near the origin. This together with the unique continuation leads to $v_{1}=v_{2}\equiv0$ in $B_{R}$.
\end{proof}

To investigate the regularity of solutions to the Helmholtz equation in a corner domain, we consider the transmission problem
\begin{equation} \label{ra}
\left\{\begin{array}{lll}
\Delta u_{\ell}+k_{\ell}^{2}u_{\ell}=0,&\quad \mbox{in}\quad\Omega_{\ell}, \\
u_{1}=u_{2}, \quad \partial_{\nu}u_{1}=\lambda\partial_{\nu}u_{2},&\quad \mbox{on}\quad\Pi_{\ell},
\end{array}\right.
\end{equation}
where $k_{\ell}$ ($\ell=1,2$) are constants satisfying $k_{1}\neq k_{2}$ and the unit normal vector $\nu$ at $\Pi_{\ell}$ is supposed to point into $\Omega_{1}$. To rewrite the system (\ref{ra}) into a divergence form, we define
\begin{equation*}
\hat{a}(\theta):=\left\{\begin{array}{ll}
1,  \quad {\rm in}~~\Omega_{1}, \\
\lambda, \quad {\rm in}~~\Omega_{2},
\end{array}\right. \quad\quad \hat{\kappa}(\theta):=\left\{\begin{array}{ll}
k_{1}^{2},\quad & {\rm in}~~\Omega_{1}, \\
\lambda k_{2}^{2},  \quad & {\rm in}~~\Omega_{2},
\end{array}\right. \quad\quad \hat{u}(r,\theta):=\left\{\begin{array}{ll}
u_{1},\quad {\rm in}~~\Omega_{1}, \\
u_{2},\quad {\rm in}~~\Omega_{2}.
\end{array}\right.
\end{equation*}
Then the transmission problem (\ref{ra}) can be equivalently written as
\begin{equation*}
\nabla\cdot(\hat{a}(\theta)\nabla \hat{u})+\hat{\kappa}(\theta)\hat{u}=0 \quad \mbox{in}~B_{R}.
\end{equation*}
By a decomposition theorem (see e.g., \cite{ES98,Pe2005, Pe2001}), one obtains
\begin{equation*}
\hat{u}=\hat{w}+\sum\limits_{j=1}^{m}c_{j}r^{\eta_{j}}\varphi_{j}(\theta)(\ln r)^{p_j} \quad \mbox{in}\quad B_{R}, \quad p_j\in\{0,1,\cdots\},
\end{equation*}
where $\hat{w}\in H^{2}(\Omega_{\ell})$ ($\ell=1,2$) and $\eta_{j}\in(0,1)$ are eigenvalues of the following positive definite Sturm-Liouville eigenvalue problem:
\begin{equation} \label{b}
\left\{\begin{array}{lll}
\varphi_{j}^{''}(\theta)+\eta_{j}^{2}\varphi_{j}(\theta)=0,\quad & \theta\in(0,3\pi/2)\cup(-\pi/2,0), \vspace{0.2cm}  \\
\varphi_{j,+}(0)=\varphi_{j,-}(0), \quad & \varphi^{'}_{j,+}(0)=\lambda \varphi_{j,-}^{'}(0), \vspace{0.2cm} \\
\varphi_{j}(3\pi/2)=\varphi_{j}(-\pi/2),\quad & \varphi_{j}^{'}(3\pi/2)=\lambda \varphi_{j}^{'}(-\pi/2).
\end{array}\right.
\end{equation}
In \eqref{b},  the subscripts '$+$' and '$-$' denote the limits from $\Omega_{1}$ and $\Omega_{2}$, respectively. It is obvious that $\eta_0=0$ is an eigenvalue with the eigenfunction $\varphi_{j,\pm}\equiv C\in \C$.
A general solution to (\ref{b}) takes the form
\begin{equation}\label{varphi}
\varphi_j(\theta)=\left\{\begin{array}{l}
A_j^{+}\cos(\eta_j\theta)+B_j^{+}\sin(\eta_j\theta),\quad \theta\in(0,3\pi/2), \\
A_j^{-}\cos(\eta_j\theta)+B_j^{-}\sin(\eta_j\theta),\quad \theta\in(-\pi/2,0),
\end{array}\right.
\end{equation}
where the non-vanishing coefficients $A_j^{\pm}$, $B_j^{\pm}$ are uniquely determined by the transmission conditions through a homogeneous 4-by-4  algebraic system. Lengthy calculations give the first positive eigenvalue (see Appendix)
\begin{equation}\label{eta1}
\eta_1=\frac{1}{\pi} \arccos\Big(-\frac{\lambda^{2}+6\lambda+1}{2(\lambda+1)^{2}}\Big)>\frac{2}{3},
\end{equation} which yields the leading singularity of $\hat{u}$ around the origin.

\begin{lemma} \label{Lema1}
For $\theta\in[0,\pi]$, we have $\varphi_j(\theta)=\varphi_j(\theta+\pi/2)$ if and only if $\eta_{j}=4N$; $\varphi_j(\theta)+\varphi_j(\theta+\pi/2)=0$ if and only if $\eta_{j}=4N+2$. Here $N\in\mathbb{N}$.
\end{lemma}
\begin{proof}
Recalling the expression of $\varphi_j(\theta)$ in (\ref{varphi}), we have
\begin{equation*}
\varphi_j(\theta+\pi/2)=A_{j}^{+}\cos(\eta_{j}(\theta+\pi/2)) +B_{j}^{+}\sin(\eta_{j}(\theta+\pi/2)), \quad \theta\in[0,\pi].
\end{equation*}
For $\eta_{j}=4N$, we obtain
\begin{equation*}
\varphi_j(\theta+\pi/2)
=A_{j}^{+}\cos(4N\theta)+B_{j}^{+}\sin(4N\theta)=\varphi_j(\theta).
\end{equation*}
If $\eta_{j}=4N+2$, then
\begin{equation*}
\varphi_j(\theta+\pi/2)
=-A_{j}^{+}\cos((4N+2)\theta)-B_{j}^{+}\sin((4N+2)\theta)=-\varphi_j(\theta).
\end{equation*}
Conversely, if $\varphi_j(\theta)=\varphi_j(\theta+\pi/2)$ for $\theta\in[0,\pi]$, then $\eta_{j}\neq4N+2$. In the following, we only need to show that the eigenvalue $\eta_j$ can't be a fractional number which implies $\eta_{j}=4N$. Setting $\theta=0$ and $\theta=\pi$ in the equality $\varphi_j(\theta)=\varphi_j(\theta+\pi/2)$ yields
\begin{equation*}
\left(\begin{array}{cc}
1-\cos(\pi\eta_{j}/2) & -\sin(\pi\eta_{j}/2) \\
\cos(\pi\eta_{j})-\cos(3\pi\eta_{j}/2)& \sin(\pi\eta_{j})-\sin(3\pi\eta_{j}/2)
\end{array}\right)
\left(\begin{array}{cc} A_{j}^{+} \\ B_{j}^{+}\end{array}\right)=\left(\begin{array}{cc} 0 \\ 0 \end{array}\right).
\end{equation*}
By simple calculation,
\begin{equation*}
\left|\begin{array}{cc}
1-\cos(\pi\eta_{j}/2) & -\sin(\pi\eta_{j}/2) \\
\cos(\pi\eta_{j})-\cos(3\pi\eta_{j}/2)& \sin(\pi\eta_{j})-\sin(3\pi\eta_{j}/2)
\end{array}\right|=2\sin(\pi\eta_{j})\big[1-\cos(\pi\eta_{j}/2)\big],
\end{equation*}
which cannot vanish when $\eta_j$ is a fractional number. Hence, $A_{j}^{+}=B_{j}^{+}=0$, which is impossible.

Similarly, if $\varphi_j(\theta)+\varphi_j(\theta+\pi/2)=0$ for $\theta\in[0,\pi]$, then $\eta_{j}\neq4N$. To show that the eigenvalue $\eta_j$ can't be a fractional number, we take $\theta=0$ and $\theta=\pi$ in the equality $\varphi_j(\theta)+\varphi_j(\theta+\pi/2)=0$. It then follows the linear system
\begin{equation*}
\left(\begin{array}{cc}
1+\cos(\pi\eta_{j}/2) & \sin(\pi\eta_{j}/2) \\
\cos(\pi\eta_{j})+\cos(3\pi\eta_{j}/2)& \sin(\pi\eta_{j})+\sin(3\pi\eta_{j}/2)
\end{array}\right)
\left(\begin{array}{cc} A_{j}^{+} \\ B_{j}^{+}\end{array}\right)=\left(\begin{array}{cc} 0 \\ 0 \end{array}\right).
\end{equation*}
In this case
the determinant of coefficient matrix is given by
$2\sin(\pi\eta_{j})\big[1+\cos(\pi\eta_{j}/2)\big]$, which
does not vanish  when $\eta_j$ is a fractional number. Hence, $\eta_j=4N+2$ for some $N\in \mathbb{N}$. 
\end{proof}

In the subsequent sections, we normalize the eigenfunctions in $L^{2}(-\pi/2,3\pi/2)$, that is, $\varphi_{0}(\theta)=1/\sqrt{2\pi}$ and
\begin{equation*}
\int_{-\pi/2}^{3\pi/2}|\varphi_{j}(\theta)|^{2}{\rm d}_{\theta}=1, \quad
\int_{-\pi/2}^{3\pi/2}\varphi_{j}(\theta) \overline{\varphi_{l}(\theta)}{\rm d}_{\theta}=\delta_{jl}:=\left\{\begin{array}{cc} 1, & \quad {\rm if}~j=l, \\
0, & \quad {\rm if}~j\neq l. \end{array}\right.
\end{equation*}
Then, we make an ansatz on the solution $\hat{u}$ to (\ref{ra}) of the form
\begin{equation}  \label{Dec}
\hat{u}(r,\theta)=\sum_{j\geq0}\alpha_{j}r^{\eta_{j}}\varphi_{j}(\theta)+\sum_{j\geq0}e_{j}(r)\varphi_{j}(\theta), \quad \alpha_j\in \C,
\end{equation}
where the second term is required to satisfy the inhomogeneous equation
\begin{equation*}
\sum_{j\geq0}\nabla\cdot \big[\hat{a}(\theta)\nabla  (e_{j}(r)\varphi_{j}(\theta))\big]=f(r,\theta),
\end{equation*}
with $f(r,\theta):=-\hat{\kappa}(\theta)\hat{u}(r,\theta)$ in $B_{R}$. Since $\hat{a}(\theta)$ is a piecewise constant function, it holds that
\begin{equation*}
\sum_{j\geq0}\Big[\frac{1}{r}(re_{j}^{'})^{'}-\frac{\eta_{j}^2}{r^{2}}e_{j}\Big]\varphi_j(\theta)=\frac{f(r,\theta)}{\hat{a}(\theta)}.
\end{equation*}
Multiplying $\overline{\varphi_l(\theta)}$ to both sides of the above equation and integrating over $(-\pi/2,3\pi/2)$ with respect to $\theta$ yields
\begin{equation*}
\frac{1}{r}(re_{j}^{'})^{'}-\frac{\eta_{j}^2}{r^{2}}e_{j}=f_{j}(r),
\end{equation*}
where
\begin{equation}\label{fj}
f_{j}(r)
=-\int_{-\pi/2}^{0}k_{2}^{2}u_{2}(r,\theta)\overline{\varphi_{j}(\theta)}{\rm d}_{\theta}
-\int_{0}^{3\pi/2}k_{1}^{2}u_{1}(r,\theta)\overline{\varphi_{j}(\theta)}{\rm d}_{\theta}.
\end{equation}
An explicit expression of $e_j$ is given by (see e.g., \cite{BFI})
\begin{equation*}
e_{j}(r)=\frac{r^{\eta_{j}}}{2\eta_{j}}\int_{r_{0}/2}^{r}f_{j}(s)s^{1-\eta_{j}}{\rm d}s-\frac{r^{-\eta_{j}}}{2\eta_{j}} \int_{0}^{r}f_{j}(s)s^{1+\eta_{j}}{\rm d}s\quad {\rm for}~j>0,~0<r_{0}<r.
\end{equation*}
In the special case $j=0$, one has
\begin{equation} \label{ej0}
\frac{1}{r}(re_{0}^{'}(r))^{'}=f_{0}(r):=-\frac{1}{\sqrt{2\pi}}\int_{-\pi/2}^{0}k_{2}^{2}u_{2}(r,\theta){\rm d}_{\theta}
-\frac{1}{\sqrt{2\pi}}\int_{0}^{3\pi/2}k_{1}^{2}u_{1}(r,\theta){\rm d}_{\theta}.
\end{equation}
Straight forward calculations yield the leading terms of $f_0$ and $e_0$.
\begin{lemma} \label{lem}
Let $u_0=u_1(O)=u_2(O)$. we have
\begin{equation*}
f_{0}(r)=-\frac{\pi}{2}\Big(k_{2}^{2}+3k_{1}^{2}\Big)\frac{u_0}{\sqrt{2\pi}}+o(1), \quad
e_{0}(r)=-\frac{\pi}{8}\Big(k_{2}^{2}+3k_{1}^{2}\Big)\frac{u_0}{\sqrt{2\pi}}r^{2}+o(r^{2}), \quad {\rm as}~ r\rightarrow0.
\end{equation*}
\end{lemma}

\section{Proof of Theorem \ref{Main}} \label{Shape}
From the coincidence of $u_1$ and $u_2$ on $\Gamma_H$,
we obtain $u_{1}=u_{2}$ in $x_{2}>H$. The unique continuation of solutions to the Helmholtz equation leads to
\begin{equation}\label{condition}
u_{1}(x_{1},x_{2})=u_{2}(x_{1},x_{2})\quad \mbox{ for all }x\in \Omega_{\Lambda_{1}}^{+}\cap\Omega_{\Lambda_{2}}^{+}.
\end{equation}
Assume on the contrary that $\Lambda_1\neq \Lambda_2$. Switching the notations for $\Lambda_1$ and $\Lambda_2$ if necessary, we only need to consider the following  cases:
\begin{itemize}
  \item Case one: there exists a corner point $O$ of $\Lambda_{1}$ such that $O\in\Omega_{\Lambda_{2}}^{+}$ (see Figure \ref{fig1});

  \item Case two: all corners of $\Lambda_1$ and $\Lambda_2$ coincide but $\Lambda_1\neq \Lambda_2$ (see Figure \ref{fig2});

  \item Case three:  there exists a corner point $O$ of $\Lambda_{2}$ lying on $\Lambda_1$, but $O$ is not a corner of $\Lambda_1$ (see Figure \ref{fig3}).
\end{itemize}
Obviously, the corners of $\Lambda_1$ and $\Lambda_2$ do not coincide completely in the first and last cases. Using coordinate translation, we suppose that the corner $O$ is located at the origin. Below we shall prove that neither of previous three cases occurs. This contraction yields $\Lambda_1=\Lambda_2$.
\subsection{Case one}\label{case-one}
Choose $R>0$ such that $B_{R} \subseteq \Omega_{\Lambda_{2}}^{+}$.
\begin{figure}[h] \label{fig1}
  \centering
  \begin{tikzpicture}[scale=0.7]
   \filldraw[gray!30,opacity=0.5] (2,4)--(3.5,4) arc (0:270:1.5);
  \filldraw[gray,opacity=0.5] (2,4)--(2,2.5) arc (270:360:1.5);
  \draw[gray, thick] (0,0) -- (2,0)-- (2,4) -- (4,4) -- (4,2) --(6,2) --(6,-1) --(8,-1)--(8,1)--(10,1)--(10,0)--(12,0);
  \draw[gray, thick, dashed]  (0,0) -- (0,1.5)-- (3,1.5) -- (3,-0.5) -- (5,-0.5) --(5,3.5) --(7,3.5) --(7,-2)--(9,-2)--(9,1.5)--(12,1.5);
  \draw[gray,thick] (2,4) circle (1.5cm); \draw (1.8,4) node{$O$};\filldraw[black] (2,4) circle (0.8pt);
  \draw (3.8,4.2) node{$\Lambda_{1}$}; \draw (3.8,-0.8) node{$\Lambda_{2}$};
  \draw (2.7,3.6) node{$B_{R}\cap\Omega_{\Lambda_{1}}^{-}$}; \draw (1.8,4.8) node{$B_{R}\cap\Omega_{\Lambda_{1}}^{+}$};
  \end{tikzpicture}
  \caption{Case one:  there exists a corner point $O$ of $\Lambda_{1}$ such that $O\in\Omega_{\Lambda_{2}}^{+}$.}
\end{figure}
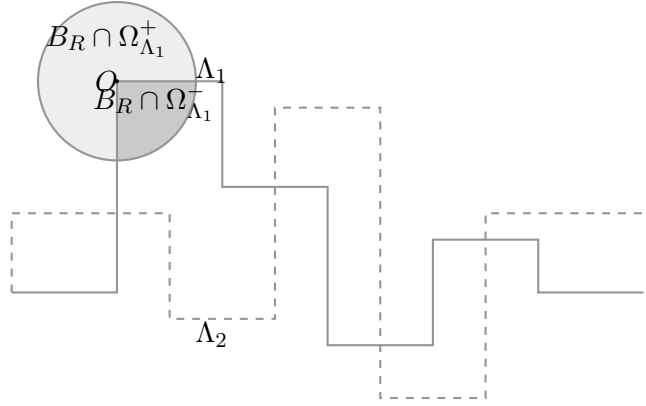
Since the corner point $O\in\Omega_{\Lambda_{2}}^{+}$ stays away from $\Lambda_{2}$, the function $u_{2}$ satisfies the Helmholtz equation with the wave number $k_{1}$ in $B_{R}$, while $u_{1}$ fulfills the Helmholtz equation with the variable potential $k_{1}^{2}(1+\rho_{1})$. Here, $\rho_{1}(x)$ is a piecewise constant function defined by
\begin{equation*}
\rho_{1}(x):=\left\{\begin{array}{lll}
 0,& \quad \mbox{in}\quad B_{R}\cap\Omega_{\Lambda_{1}}^{+},  \vspace{0.2cm} \\
 (\frac{k_{2}}{k_{1}})^{2}-1,& \quad \mbox{in}\quad B_{R}\cap\Omega_{\Lambda_{1}}^{-}.
\end{array}\right.
\end{equation*}
Recalling the transmission conditions in (\ref{a}), we find that the pair $(u_{1},u_{2})$ is a solution to
\begin{equation*}
\left\{\begin{array}{lll}
\Delta u_{1}+k_{1}^{2}(1+\rho_{1}(x))u_{1}=0,\quad &\mbox{in}\quad B_{R},\vspace{0.1cm}\\
\Delta u_{2}+k_{1}^{2}u_{2}=0,\quad &\mbox{in}\quad B_{R}, \vspace{0.1cm}\\
u_{1}=u_{2}, \quad \lambda\frac{\partial u_{1}^{-}}{\partial\nu}=\frac{\partial u_{2}}{\partial\nu}, \quad &\mbox{on}\quad  B_{R}\cap\Lambda_{1}.
\end{array}\right.
\end{equation*}
Here, the symbol $(\cdot)^-$ denotes the limit from $\Omega_{\Lambda_{1}}^{-}$.
Applying Lemma \ref{Th2}, we obtain $u_{1}=0$ in $B_{R}$ and thus $u_{1}=0$ in $\mathbb{R}^{2}$, which is impossible (see \cite{XH}).

\subsection{Case two}
The corners of $\Lambda_{1}$ and $\Lambda_{2}$ coincide (see Figure \ref{fig2}), implying that $\Lambda_1$ and $\Lambda_2$ have the same height and also the same grooves but with different opening directions. This section relies on ingenious analysis on the regularity of solutions to the Helmholtz equation in a corner domain. We refer to \cite{Pe2001} for an overview of the interface problem of the Laplacian equation.
\begin{figure}[h] \label{fig2}
  \centering
  \begin{tikzpicture}[scale=0.7]
    \draw[gray, thick] (0,0) -- (2,0)-- (2,4) -- (4,4) -- (4,0) --(5.5,0) --(5.5,4) --(7,4)--(7,0)--(9,0);
  \draw[gray, thick, dashed] (0,4)-- (2,4) -- (2,0) -- (4,0) --(4,4) --(5.5,4) --(5.5,0)--(7,0)--(7,4)--(9,4);
  \draw[gray,thick] (2,4) circle (1cm); \draw (1.8,4.2) node{$O$};\filldraw[black] (2,4) circle (0.8pt);
  \draw (3.4,4.2) node{$\Lambda_{1}$}; \draw (3.4,0.2) node{$\Lambda_{2}$};
  \end{tikzpicture}
  \caption{Case two: corners of $\Lambda_1$ and $\Lambda_2$ are identical but $\Lambda_1\neq \Lambda_2$.}
\end{figure}
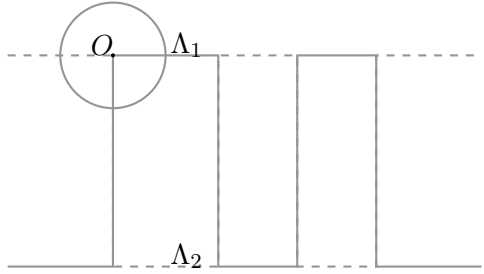

Choose a corner point $O\in\Lambda_{1}\cap\Lambda_{2}$ and $R>0$ sufficiently small such that the disk $B_{R}:=\{x\in \mathbb{R}^{2}:|x|<R\}$ does not contain other corners. We can conclude from Proposition \ref{prop} that $u_{1},u_{2}\in H^{1+s}(B_{R})$ ($0\leq s<1/2$) fulfill the system
\begin{equation}\label{uj}
\left\{\begin{array}{lll}
\nabla\cdot(a(\theta)\nabla u_1)+\kappa(\theta)u_{1}=0,\quad &\mbox{in}\quad B_{R},\vspace{0.1cm} \\
\nabla\cdot(a(\theta+\pi/2)\nabla u_{2})+\kappa(\theta+\pi/2)u_{2}=0,\quad & \mbox{in}\quad B_R,
\end{array}\right.
\end{equation}
where
\begin{equation*}
a(\theta):=\left\{\begin{array}{lll}
1, \quad & \mbox{if}\quad \theta\in (0,3\pi/2),\\
\lambda, \quad &\mbox{if}\quad \theta\in (-\pi/2,0),
\end{array}\right. \quad\quad
\kappa(\theta):=\left\{\begin{array}{lll}
k_1^2, \quad &\mbox{if} \quad \theta\in (0, 3\pi/2),\\
\lambda k_2^2, \quad & \mbox{if}\quad \theta\in (-\pi/2,0),
\end{array}\right.
\end{equation*}
and $a(\theta\pm2\pi)=a(\theta)$, $\kappa(\theta\pm2\pi)=\kappa(\theta)$.
It is obvious that $u_2$ coincides with $u_1$ after a rotation about the angle $\pi/2$, that is, $u_2(r, \theta)=u_1(r,\theta+\pi/2)$. In Lemma \ref{Lema2} below, we shall derive a more explicit expression of $u_{\ell}$ $(\ell=1,2)$ under the condition \eqref{condition}.

\begin{lemma} \label{Lema2}
Let $u_{1},u_{2}\in H^{1+s}(B_{R})$ ($0\leq s<1/2$) be solutions to \eqref{uj}. If
\ben
u_1(r,\theta)=u_2(r, \theta)\quad {\rm for~all} \quad \theta\in (0,\pi),\, r\in[0,R),
\enn
then
\begin{equation} \label{equa}
u_{\ell}(r,\theta)=\sum_{n,m\in\mathbb{N}:n+m\geq0}a_{n,m}^{(\ell)}r^{2(n+m)} \psi^{(\ell)}_{2n}(\theta),\quad \ell=1,2
\end{equation}
where $\psi_{2n}^{(1)}(\theta)$ is the normalized eigenfunction of \eqref{b} corresponding to the eigenvalue $\eta=2n$ and $\psi_{2n}^{(2)}(\theta)=\psi_{2n}^{(1)}(\theta+\pi/2)$. 
\end{lemma}
\begin{proof}
To prove (\ref{equa}), it suffices to verify for all $l\in\mathbb{N}$ that
\begin{equation} \label{equa2}
u_{\ell}(r,\theta)=\sum\limits_{0\leq n+m\leq l}a^{(\ell)}_{n,m}r^{2(n+m)}\psi_{2n}^{(\ell)}(\theta)+o(r^{2l}),
\quad {\rm as}\quad r\rightarrow 0.
\end{equation}
Recalling (\ref{Dec}) and Lemma \ref{lem}, we have
\begin{equation} \label{Dec1}
u_{\ell}(r,\theta)=u_{0}+\sum_{j\geq1}\alpha^{(\ell)}_{j}r^{\eta_{j}}\varphi_{j}^{(\ell)}(\theta) +e_{0,0}^{(\ell)}(r)\varphi_{0}^{(\ell)}(\theta)+\sum_{j\geq1}e_{j,0}^{(\ell)}(r)\varphi_{j}^{(\ell)}(\theta),\quad \quad \ell=1,2,
\end{equation}
where $\varphi_{j}^{(1)}(\theta):=\varphi_{j}(\theta)$ are normalized eigenfunctions, $\varphi_{j}^{(2)}(\theta):=\varphi_{j}^{(1)}(\theta+\pi/2)$ and
\begin{equation} \label{e_j}
e_{j,0}^{(\ell)}(r)=\frac{r^{\eta_{j}}}{2\eta_{j}}\int_{r_{0}/2}^{r}f_{j,0}^{(\ell)}(s) s^{1-\eta_{j}}{\rm d}s-\frac{r^{-\eta_{j}}}{2\eta_{j}} \int_{0}^{r}f_{j,0}^{(\ell)}(s)s^{1+\eta_{j}}{\rm d}s,\quad {\rm for}~j>0,~\ell=1,2.
\end{equation}
Here the functions $f_{j,0}^{(\ell)}$ with $\ell=1,2$ are defined analogously to \eqref{fj} and $0<r_0<r$.
By \eqref{eta1}, we know that $\eta_{j}>2/3$ for $j\geq1$, which together with $e_{j,0}^{(\ell)}(r)=o(r)$ ($\ell=1,2$) implies that (\ref{equa2}) holds with $l=n+m=0$ and $a^{(1)}_{0,0}=a^{(2)}_{0,0}=\sqrt{2\pi}u_{0}$. 

\bigskip
{\bf Step 1}: Prove that (\ref{equa2}) holds for $l=1$. It is obvious that if $l=n+m=1$ for some $n,m\in\mathbb{N}$, then $n=0$, $m=1$ or $n=1$, $m=0$. Hence, it suffices to prove
\begin{equation*}
u_{\ell}(r,\theta)=a_{0,0}^{(\ell)}\psi_{0}^{(\ell)}(\theta)+\big[a_{0,1}^{(\ell)}\psi_{0}^{(\ell)}(\theta) +a^{(\ell)}_{1,0}\psi_{2}^{(\ell)}(\theta)\big]r^{2}+o(r^{2}),\quad \mbox{as}\quad r\rightarrow0,\; \ell=1,2,
\end{equation*}
with some $a^{(\ell)}_{0,1}$, $a^{(\ell)}_{1,0}\in\mathbb{C}$ for $\ell=1,2$. Recalling from the definition of $e^{(\ell)}_{j,0}$ ($j\geq0$, $\ell=1,2$) in (\ref{e_j}), we obtain
\be\label{ej}
e^{(\ell)}_{j,0}(r)=\left\{\begin{array}{lll}
\frac{\sqrt{2\pi}}{4-\eta_j^2}\,d_{j,0}\,u_{0}\, r^2+o(r^{3}), &\mbox{if} \quad \eta_j\neq 2, \vspace{0.2cm}\\
\frac{\sqrt{2\pi}}{4}\,d_{j,0}\,u_{0}\,r^2\ln r+o(r^{3}), & \mbox{if} \quad \eta_j=2,
\end{array}\right.\qquad {\rm as}~r\rightarrow 0,
\en
where $d_{j,0}\in \C$ are given by
\be \label{dj}
d_{j,0}:=-\left[k^2_2\int^{0}_{-\pi/2}\psi_{0}^{(1)}(\theta)\overline{\varphi_{j}^{(1)}(\theta)}\,{\rm d}_\theta+k_1^2\int_{0}^{3\pi/2}\psi_{0}^{(1)}(\theta)\overline{\varphi_{j}^{(1)}(\theta)}\,{\rm d}_{\theta}\right],\quad \eta_j\geq 0.
\en
Hence, it follows from \eqref{Dec1} that
\ben
u_{\ell}(r,\theta)=u_{0}+\sum_{0<\eta_{j}<2}\alpha^{(\ell)}_{j}r^{\eta_{j}}\varphi^{(\ell)}_{j}(\theta)+o(r^{l_0}), 
\enn
where $l_0=\max\{\eta_j: 0<\eta_{j}<2\}$. Recalling $u_{1}(r,\theta)=u_{2}(r,\theta)$, $\partial_{\theta}u_{1}(r,\theta) =\partial_{\theta}u_{2}(r,\theta)$ ($\theta\in[0,\pi]$), we obtain  \begin{equation*}
\alpha_{j}^{(1)}\varphi_{j}^{(1)}(\theta)=\alpha^{(2)}_{j}\varphi_{j}^{(2)}(\theta), \quad \alpha^{(1)}_{j}\big[\varphi_{j}^{(1)}(\theta)\big]^{'} =\alpha^{(2)}_{j}\big[\varphi_{j}^{(2)}(\theta)\big]^{'},\quad  \forall~\theta \in(0,\pi),~\eta_{j}\in(0,2),
\end{equation*}
which we can be rewritten as the linear system
\begin{equation*}
\left(\begin{array}{cc}
A_{j}^{+}\cos(\eta_{j}\theta)+B_{j}^{+}\sin(\eta_{j}\theta) & -A_{j}^{+}\cos(\eta_{j}(\theta+\frac{\pi}{2}))-B_{j}^{+}\sin(\eta_{j}(\theta+\frac{\pi}{2}))\\
B_{j}^{+}\cos(\eta_{j}\theta)-A_{j}^{+}\sin(\eta_{j}\theta) & A_{j}^{+}\sin(\eta_{j}(\theta+\frac{\pi}{2}))-B_{j}^{+}\cos(\eta_{j}(\theta+\frac{\pi}{2}))
\end{array}\right)
\left(\begin{array}{cc} \alpha^{(1)}_{j} \\ \alpha^{(2)}_{j}\end{array}\right)=\left(\begin{array}{cc} 0 \\ 0 \end{array}\right).
\end{equation*}
Since the determinant of coefficient matrix is $\big[(A_{j}^{+})^{2}+(B_{j}^{+})^{2}\big]\sin\left(\frac{\pi}{2}\eta_{j}\right)>0$, we obtain $\alpha_{j}^{(1)}=\alpha^{(2)}_{j}=0$ for $0<\eta_{j}<2$.
It then follows from \eqref{Dec1} that
\begin{equation*}
u_{\ell}(r,\theta)=u_{0}+\sum_{2\leq\eta_{j}<4}\alpha^{(\ell)}_{j}r^{\eta_{j}}\varphi^{(\ell)}_{j}(\theta) +\sum_{\eta_j\geq 0}e^{(\ell)}_{j,0}(r)\varphi_j^{(\ell)}(\theta)+o(r^{l_1}),\qquad{\rm as}~r\rightarrow 0,
\end{equation*}
where $l_1=\max\{\eta_j: 2<\eta_{j}<4\}$. Hence,
\be\label{u1}
u_{\ell}(r,\theta)&=&u_{0}+a^{(\ell)}_{1,0}r^2\psi^{(\ell)}_{2}(\theta)+\sum_{2<\eta_{j}<4}\alpha^{(\ell)}_{j}r^{\eta_{j}}\varphi^{(\ell)}_{j}(\theta) +\frac{\sqrt{2\pi}}{4}D_{2,0}\,u_{0} \,r^2\ln r\; \psi_{2}^{(\ell)}(\theta) \\ \nonumber
&&+\,u_{0}\, r^2\sum_{\eta_j\neq 2}\frac{\sqrt{2\pi}\,d_{j,0}}{4-\eta_j^2}\,\varphi^{(\ell)}_j(\theta)+o(r^{l_1}), \quad\quad {\rm as}~r\rightarrow 0,
\en
where $a^{(\ell)}_{1,0}=\alpha^{(\ell)}_{j}$, $D_{2,0}=d_{j,0}$ for $\eta_j=2$. 
Equating the coefficients of the terms $r^2$ and $r^2\ln r$ yields
\ben
&&D_{2,0}u_{0}\;\big[\psi_{2}^{(1)}(\theta)-\psi_{2}^{(2)}(\theta)\big]=0, \\
&&\big[a^{(1)}_{1,0}\psi_{2}^{(1)}(\theta)-a^{(2)}_{1,0}\psi_{2}^{(2)}(\theta)\big]+
\, \sum_{\eta_j\neq 2} \frac{\sqrt{2\pi}d_{j,0}u_{0}}{4-\eta_j^2}\,\big[\varphi_{j}^{(1)}(\theta)-\varphi_{j}^{(2)}(\theta)\big]=0,
\enn
for all $\theta\in(0,\pi)$. Since $\psi_{2}^{(2)}(\theta)=-\psi_{2}^{(1)}(\theta)$, by linear independence of trigonometric functions, we conclude that
\begin{equation*}
D_{2,0}\,u_{0}=0, \quad a^{(1)}_{1,0}+a^{(2)}_{1,0}=0 \quad\quad {\rm and} \quad\quad d_{j,0}\,u_{0}=0 ~~{\rm if}~~\varphi_{j}^{(1)}(\theta)\neq\varphi_{j}^{(2)}(\theta),\; \eta_j\neq 2.
\end{equation*}
If $\varphi_{j}^{(1)}(\theta)=\varphi_{j}^{(2)}(\theta)$, we have $\eta_{j}=4N$ by Lemma \ref{Lema1} and
\begin{equation*}
d_{j,0}=\left\{\begin{array}{lll}
0, && \mbox{if}\quad \eta_{j}=4N,\;N\neq0,\\
-\frac{1}{4}(3k_{1}^{2}+k_{2}^{2}), && \mbox{if}\quad \eta_{j}=0\;({\rm i.e.}~j=0).
\end{array}\right.
\end{equation*}
This implies that the terms with
$j\neq 0$ in the following summation all vanish, i.e.,
\begin{equation*}
r^2\,u_0\sum_{\eta_j\neq 2} \frac{\sqrt{2\pi}\,d_{j,0}}{4-\eta_j^2}\,\varphi^{(\ell)}_j(\theta) =r^2\,u_0\frac{\sqrt{2\pi}\,d_{0,0}}{4}\,\varphi^{(\ell)}_0(\theta).
\end{equation*}
Inserting these results into (\ref{u1})  yields as $r\rightarrow 0$ that
\ben
u_{\ell}(r,\theta)=\sum\limits_{0\leq n+m\leq 1}a_{n,m}^{(\ell)}r^{2(n+m)}\psi_{2n}^{(\ell)}(\theta) +\sum\limits_{2<\eta_{j}<4}\alpha_{j}^{(\ell)}r^{\eta_{j}}\varphi_{j}^{(\ell)}(\theta) +\sum\limits_{\eta_{j}\geq0}e_{j,1}^{(\ell)}(r)\varphi_{j}^{(\ell)}(\theta)+o(r^{l_{1}})
\enn
where $a^{(\ell)}_{0,0}=\sqrt{2\pi}u_{0}$, $a^{(\ell)}_{0,1}=\sqrt{2\pi}\,d_{0,0}\,u_{0}/4$, $a^{(1)}_{1,0}=-a^{(2)}_{1,0}$.
Further, we have
$a_{0,1}^{(\ell)}=a_{0,0}^{(\ell)}\, d_{0,0}/4$ and
\ben
&&a_{0,0}^{(\ell)}\,d_{j,0}=0\quad{\rm for}~\eta_{j}\neq 0; \quad
a_{n,m}^{(1)}\psi_{2n}^{(1)}(\theta)=a_{n,m}^{(2)}\psi_{2n}^{(2)}(\theta)\quad\mbox{for all}\; \,0\leq n+m\leq 1,\\
&&
e^{(\ell)}_{j,1}(r)=e^{(\ell)}_{j,0}(r)-\left\{\begin{array}{lll}
\frac{\sqrt{2\pi}}{4-\eta_j^2}\,d_{j,0}\,u_{0}\, r^2, &\mbox{if} \quad \eta_j\neq 2, \vspace{0.2cm}\\
\frac{\sqrt{2\pi}}{4}\,d_{j,0}\,u_{0}\,r^2\ln r, & \mbox{if} \quad \eta_j=2.
\end{array}\right.
\enn
It is seen from \eqref{ej} that $e^{(\ell)}_{j,1}(r)=o(r^3)$. This finishes the Step 1.
\bigskip

{\bf Step 2}: Induction arguments. We make an induction hypothesis that for some $N\geq1$,
\begin{align} \label{u1N}
\left\{\begin{array}{lll}
u_{\ell}(r,\theta)=\sum\limits_{0\leq n+m\leq N}a_{n,m}^{(\ell)}r^{2(n+m)}\psi_{2n}^{(\ell)}(\theta) +\sum\limits_{2N<\eta_{j}<2N+2}\alpha_{j}^{(\ell)}r^{\eta_{j}}\varphi_{j}^{(\ell)}(\theta) \vspace{0.2cm}   \\
\quad\quad\quad\quad\quad+\sum\limits_{\eta_{j}\geq0}e_{j,N}^{(\ell)}(r)\varphi_{j}^{(\ell)}(\theta)+o(r^{l_{N}});  \vspace{0.3cm} \\
a_{n,m}^{(\ell)}=\frac{a_{n,m-1}^{(\ell)}\, D_{2n,2n}}{(2N)^{2}-(2n)^{2}}, 
~\forall \,n+m=N,\,0\leq n\leq N-1; \vspace{0.2cm} \\
a_{n,m}^{(\ell)}\,d_{j,2n}=0, \quad {\rm for}~\eta_{j}\neq 2n,~\forall \,0\leq n+m\leq N-1;  \vspace{0.2cm} \\
a_{n,m}^{(1)}\psi_{2n}^{(1)}(\theta)=a_{n,m}^{(2)}\psi_{2n}^{(2)}(\theta), \quad \forall \,0\leq n+m\leq N, 
\end{array}\right.
\end{align}
where $e^{(\ell)}_{j,N}(r)$ ($\ell=1,2$) is defined as (\ref{ej0}), (\ref{e_j}) with $f_{j,0}^{(\ell)}$ replaced by $f_{j,N}^{(\ell)}$:
\begin{align*}
f_{j,N}^{(\ell)}(r)=&-\int_{0}^{3\pi/2}k_{1}^{2}\Big[u_{\ell}(r,\theta)-\sum_{0\leq n+m\leq N-1}a_{n,m}^{(\ell)}r^{2(n+m)}\psi^{(\ell)}_{2n}(\theta)\Big] \overline{\varphi_{j}^{(\ell)}(\theta)}{\rm d}_{\theta} \\
&-\int_{-\pi/2}^{0}k_{2}^{2}\Big[u_{\ell}(r,\theta)-\sum_{0\leq n+m\leq N-1}a_{n,m}^{(\ell)}r^{2(n+m)}\psi_{2n}^{(\ell)}(\theta)\Big] \overline{\varphi_{j}^{(\ell)}(\theta)}{\rm d}_{\theta};
\end{align*}
$l_{N}:=\max\{\eta_j: 2N<\eta_{j}<2N+2\}$;
\begin{align}\nonumber
d_{j,2n}=&-\left[k^{2}_{2}\int^{0}_{-\pi/2}\psi_{2n}^{(1)}(\theta)\overline{\varphi_{j}^{(1)}(\theta)}\,{\rm d}_{\theta} +k_{1}^{2}\int_{0}^{3\pi/2}\psi_{2n}^{(1)}(\theta)\overline{\varphi_{j}^{(1)}(\theta)}\,{\rm d}_{\theta}\right] \vspace{0.2cm}\\ \label{dj2n1}
=&\left\{\begin{array}{lll}
-k_{1}^{2}+(k_{1}^{2}-k_{2}^{2})\int^{0}_{-\pi/2}|\psi_{2n}^{(1)}(\theta)|^{2}\,{\rm d}_{\theta}, && \mbox{if} \quad \eta_{j}=2n,\vspace{0.2cm}\\
(k_{1}^{2}-k_{2}^{2})\int^{0}_{-\pi/2}\psi_{2n}^{(1)}(\theta)\overline{\varphi_{j}^{(1)}(\theta)}\,{\rm d}_{\theta}, &&\mbox{if}\quad\eta_{j}\neq2n,
\end{array}\right.
\end{align}
for $0\leq n\leq N-1$; $D_{2n,2n}:=d_{j,2n}$ when $\eta_{j}=2n$.

Note that the above induction hypothesis with $N=1$ has been proved in Step one. Now we want to prove that (\ref{u1N}) holds for $N+1$. By the definition of $e^{(\ell)}_{j,N}$, straightforward calculations show that
\be\label{ejnl}
e^{(\ell)}_{j,N}(r)=\left\{\begin{array}{lll}
\frac{r^{2N+2}}{(2N+2)^{2}-\eta_{j}^2}\sum\limits_{n+m=N}a_{n,m}^{(\ell)}d_{j,2n}+o(r^{2N+3}), &\mbox{if}\quad \eta_{j}\neq 2N+2,\vspace{0.2cm}\\
\frac{a_{N,0}^{(\ell)}\,D_{2N+2,2N}}{4N+2}r^{2N+2}\ln r+o(r^{2N+3}), & \mbox{if}\quad \eta_{j}=2N+2.
\end{array}\right.
\en
Here $D_{2N+2,2N}:=d_{j,2N}$ with $\eta_{j}=2N+2$ and $d_{j, 2N}$ is defined analogously by \eqref{dj2n1}.

Using the relations $u_{1}(r,\theta)=u_{2}(r,\theta)$, $\partial_{\theta}u_{1}(r,\theta)=\partial_{\theta}u_{2}(r,\theta)$ ($\theta\in[0,\pi]$), we deduce from the expressions of $u_l$ in \eqref{u1N} that
\begin{equation*}
\alpha_{j}^{(1)}\varphi_{j}^{(1)}(\theta)=\alpha_{j}^{(2)}\varphi_{j}^{(2)}(\theta), \quad \alpha_{j}^{(1)}\big[\varphi_{j}^{(1)}(\theta)\big]^{'} =\alpha_{j}^{(2)}\big[\varphi_{j}^{(2)}(\theta)\big]^{'},\quad  \forall \,\theta \in(0,\pi),~\eta_{j}\in(2N,2N+2).
\end{equation*}
Similarly, we can obtain an equation system about the unknowns $\alpha_{j}^{(1)}$ and $\alpha_{j}^{(2)}$, where the determinant of coefficient matrix is still not equal to zero for $2N<\eta_{j}<2N+2$. Consequently, we achieve that $\alpha_{j}^{(1)}=\alpha_{j}^{(2)}=0$ for $2N<\eta_{j}<2N+2$. Inserting this into (\ref{u1N}) gives
\begin{align*}
u_{\ell}(r,\theta)=&\sum_{0\leq n+m\leq N}a_{n,m}^{(\ell)}r^{2(n+m)} \psi_{2n}^{(\ell)}(\theta)+\sum_{2N+2\leq\eta_{j}<2N+4}\alpha_{j}^{(\ell)}r^{\eta_{j}}\varphi_{j}^{(\ell)}(\theta)\\
&+\sum_{\eta_{j}\geq0}e_{j,N}^{(\ell)}(r)\varphi_{j}^{(\ell)}(\theta)+o(r^{l_{N}}),\quad \ell=1,2.
\end{align*}
Using the relations in (\ref{ejnl}), we can obtain
\begin{align*}
u_{\ell}(r,\theta)=&\sum_{0\leq n+m\leq N}a_{n,m}^{(\ell)}r^{2(n+m)}\psi_{2n}^{(\ell)}(\theta)+r^{2N+2}\sum_{0\leq n\leq N-1}^{n+m=N+1}a_{n,m}^{(\ell)}\psi_{2n}^{(\ell)}(\theta)
+a_{N+1,0}^{(\ell)}r^{2N+2}\psi_{2N+2}^{(\ell)}(\theta)\\
&+\frac{a_{N,0}^{(\ell)}\,D_{2N+2,2N}}{4N+2}r^{2N+2}\ln r \,\psi_{2N+2}^{(\ell)}(\theta)
+\sum_{\eta_{j}\neq 2N+2}\frac{a_{N,0}^{(\ell)}\,d_{j,2N}}{(2N+2)^{2}-\eta_{j}^2}r^{2N+2}\varphi_{j}^{(\ell)}(\theta) \\
&+\sum_{2N+2<\eta_{j}<2N+4}\alpha_{j}^{(\ell)}r^{\eta_{j}}\varphi_{j}^{(\ell)}(\theta)+o(r^{l_{N+1}}),\quad \ell=1,2.
\end{align*}
Here, $a_{N+1,0}^{(\ell)}:=\alpha_{j}^{(\ell)}$ for $\eta_{j}=2N+2$, $l_{N+1}:=\max\{\eta_j: 2N+2<\eta_{j}<2N+4\}$ and
\begin{equation} \label{i1}
a_{n,m}^{(\ell)}=\frac{a_{n,m-1}^{(\ell)}\, D_{2n,2n}}{(2N+2)^{2}-(2n)^{2}}, \quad \forall \,0\leq n\leq N-1,\,n+m=N+1. 
\end{equation}
Applying the induction hypothesis $a_{n,m}^{(1)}\psi_{2n}^{(1)}(\theta)=a_{n,m}^{(2)}\psi_{2n}^{(2)}(\theta)$ for all $0\leq n+m\leq N$ into (\ref{i1}), we have
\begin{equation} \label{i2}
a_{n,m}^{(1)}\psi_{2n}^{(1)}(\theta)=a_{n,m}^{(2)}\psi_{2n}^{(2)}(\theta), \quad \forall \,0\leq n\leq N-1,\,n+m=N+1.
\end{equation}

Comparing the expressions of $u_{1}$ and $u_{2}$ and using the fact that $u_{1}=u_{2}$ for all $\theta\in(0,\pi)$ yields
\begin{equation*}
a_{N,0}^{(1)}\,D_{2N+2,\,2N}\,\psi_{2N+2}^{(1)}(\theta)=a_{N,0}^{(2)}\,D_{2N+2,\,2N}\, \psi_{2N+2}^{(2)}(\theta),
\end{equation*}
and
\begin{align*}
&a_{N+1,0}^{(1)}\,\psi_{2N+2}^{(1)}(\theta)+\sum_{\eta_{j}\neq2N+2}\frac{a_{N,0}^{(1)}\,d_{j,2N}}{(2N+2)^{2}-\eta_{j}^2}\varphi_{j}^{(1)}(\theta) \\
=&a_{N+1,0}^{(2)}\,\psi_{2N+2}^{(2)}(\theta)+\sum_{\eta_{j}\neq2N+2}\frac{a_{N,0}^{(2)}\,d_{j,2N}}{(2N+2)^{2}-\eta_{j}^2}\varphi_{j}^{(2)}(\theta).
\end{align*}
Since $a_{N,0}^{(1)}=(-1)^{N}a_{N,0}^{(2)}$, $\psi_{2N+2}^{(2)}(\theta)=(-1)^{N+1}\psi_{2N+2}^{(1)}(\theta)$, we conclude that
\begin{equation*}
a_{N,0}^{(\ell)}\,D_{2N+2,2N}\,\psi_{2N+2}^{(\ell)}(\theta)=0,
\end{equation*}
and
\begin{equation*}
\big[a_{N+1,0}^{(1)}-(-1)^{N+1}a_{N+1,0}^{(2)}\big]\psi_{2N+2}^{(1)}(\theta)+\sum_{\eta_{j}\neq2N+2}\frac{a_{N,0}^{(1)}d_{j,2N}}{(2N+2)^{2}-\eta_{j}^2} \left[\varphi_{j}^{(1)}(\theta)-(-1)^{N}\varphi_{j}^{(2)}(\theta)\right]=0.
\end{equation*}
Using Lemma \ref{Lema1} and the linear independence of trigonometric functions, we conclude that
\begin{equation} \label{i3}
a_{N+1,0}^{(1)}\psi_{2N+2}^{(1)}(\theta)=a_{N+1,0}^{(2)}\psi_{2N+2}^{(2)}(\theta),  
\end{equation}
and
\begin{equation*}
a_{N,0}^{(1)}\,d_{j,2N}=\left\{\begin{array}{lll}
0, && \mbox{if}\quad \varphi_j^{(1)}(\theta)\neq\varphi_j^{(2)}(\theta),~N\mbox{ is an even number}, \\
0, && \mbox{if}\quad \varphi_j^{(1)}(\theta)+\varphi_j^{(2)}(\theta)\neq0,~N\mbox{ is an odd number}.
\end{array}\right.
\end{equation*}
Recalling Lemma \ref{Lema1} and the definition of $d_{j,2N}$, we find that
\begin{equation*}
d_{j,2N}=\left\{\begin{array}{lll}
0, && \mbox{if}\quad \eta_{j}=4l \mbox{ and } N \mbox{ is an even number},~l\neq N/2,\\
0, && \mbox{if}\quad \eta_{j}=4l+2 \mbox{ and } N \mbox{ is an odd number},~l\neq(N-1)/2.
\end{array}\right.
\end{equation*}
Based on the above results, we conclude that
\begin{equation*}
a_{N,0}^{(\ell)}\,d_{j,2N}=0, \quad {\rm for}~\eta_{j}\neq 2N,~\ell=1,2.  
\end{equation*}
Combining the previous equalities with the following two induction hypothesis
\begin{equation*}
\left\{\begin{array}{lll}
a_{n,m}^{(\ell)}=\frac{a_{n,m-1}^{(\ell)}\, D_{2n,2n}}{(2N)^{2}-(2n)^{2}}, \quad 
~\forall \,n+m=N,\,0\leq n\leq N-1, \vspace{0.2cm} \\
a_{n,m}^{(\ell)}\,d_{j,2n}=0, \quad {\rm for}~\eta_{j}\neq 2n,~\forall \,0\leq n+m\leq N-1,  
\end{array}\right.
\end{equation*}
we find that
\begin{equation} \label{i4}
a_{n,m}^{(\ell)}\,d_{j,2n}=0, \quad {\rm for}~\eta_{j}\neq 2n,~\forall \,0\leq n\leq N,\,n+m=N.   
\end{equation}
Hence,
\begin{align} \label{i5}
u_{\ell}(r,\theta)=&\sum\limits_{0\leq n+m\leq N+1}a_{n,m}^{(\ell)}r^{2(n+m)}\psi_{2n}^{(\ell)}(\theta)
+\sum\limits_{2N+2<\eta_{j}<2N+4}\alpha_{j}^{(\ell)}r^{\eta_{j}}\varphi_{j}^{(\ell)}(\theta) \\ \nonumber
&+\sum\limits_{\eta_{j}\geq0}e_{j,N+1}^{(\ell)}(r)\varphi_{j}^{(\ell)}(\theta)+o(r^{l_{N+1}}), \quad \ell=1,2,
\end{align}
where $e^{(\ell)}_{j,N+1}$ is defined in the same way as $e^{(\ell)}_{j,N}$, $D_{2N,2N}$ equals to $d_{j,2N}$ when $\eta_{j}=2N$ and
\begin{equation} \label{i7}
a_{N,1}^{(\ell)}=\frac{a_{N,0}^{(\ell)}\,D_{2N,2N}}{(2N+2)^{2}-(2N)^2},  \quad \ell=1,2.
\end{equation}
Then, the relation $a_{N,0}^{(1)}\psi_{2N}^{(1)}(\theta)=a_{N,0}^{(2)}\psi_{2N}^{(2)}(\theta)$ gives that
\begin{equation} \label{i8}
a_{N,1}^{(1)}\psi_{2N}^{(1)}(\theta)=a_{N,1}^{(2)}\psi_{2N}^{(2)}(\theta).
\end{equation}
Therefore, relations (\ref{i1})--(\ref{i8}) imply that (\ref{u1N}) still holds for $N+1$.

{\bf Step 3:} By the induction argument, we know that (\ref{u1N}) holds for any $N\in\mathbb{N}$£¬ which implies  (\ref{equa2}) for all $l\in\mathbb{N}$. Hence, the proof of (\ref{equa}) is complete.
\end{proof}

By Lemma \ref{Lema2}, we have
\begin{align*}
u_{1}(r,\theta)
=\left\{\begin{array}{lll}
\sum\limits_{n+m\geq0}a_{n,m}^{(1)}r^{2(n+m)}[A_{n}^{-}\cos(2n\theta)+B_{n}^{-}\sin(2n\theta)], \quad  \theta\in(-\pi/2,0),\vspace{0.2cm}\\
\sum\limits_{n+m\geq0}a_{n,m}^{(1)}r^{2(n+m)}[A_{n}^{+}\cos(2n\theta)+B_{n}^{+}\sin(2n\theta)], \quad  \theta\in(0,3\pi/2).
\end{array}\right.
\end{align*}
Now,  using the transmission condition of $u_1$ on $\Pi_{\ell}$ one can repeat the proof in the proof of
Lemma \ref{Th1} to obtain $u_1\equiv 0$ around $O$, which is impossible. This excludes the case two.

\subsection{Case three}

Assume there exists a corner $O$ of $\Lambda_2$ such that $O\in\Lambda_1$, but $O$ is not a corner point of $\Lambda_1$. Without loss of generality, we suppose that $O$ is located on a vertical line segment of $\Lambda_1$ (see Figure \ref{fig3}).
\begin{figure}[h] \label{fig3}
  \centering
  \begin{tikzpicture}[scale=0.75]
    \draw[gray, thick] (-2,2) -- (0,2)-- (0,-2) -- (3,-2) -- (3,1) --(5,1) --(5,-1) --(7,-1)--(7,1.5)--(9,1.5);
  \draw[gray, thick, dashed] (-2,2) -- (0,2) -- (0,0)-- (3,0) -- (3,1) -- (5,1) --(5,0.3) --(7,0.3) --(7,1.5)--(9,1.5);
  \draw[gray,thick] (0,0) circle (1cm); \draw (-0.2,0.2) node{$O$};\filldraw[black] (0,0) circle (0.8pt);
  \draw (-1,2.2) node{$\Lambda_{1}$}; \draw (2,0.2) node{$\Lambda_{2}$};
  \end{tikzpicture}
  \caption{Case three: $O\in\Lambda_1\cap\Lambda_2$ is a corner of $\Lambda_2$ but not a corner of $\Lambda_1$.}
\end{figure}
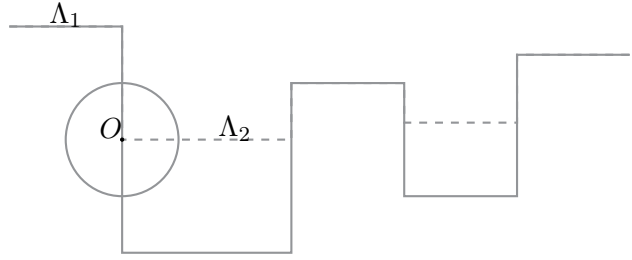 Choose $R>0$ sufficiently small such that the disk $B_{R}$ does not contain any other corners. We can see that $u_{1},u_{2}\in H^{1+s}(B_{R})$ ($0\leq s<1/2$) are solutions to the systems
\begin{align}  \label{a1}
&\left\{\begin{array}{lll}
 \Delta u_{1}+k_{1}^{2}u_{1}=0, & \quad \mbox{in} \quad \theta\in[0,\pi/2)\cup(3\pi/2,2\pi], \vspace{0.2cm} \\
 \Delta u_{1}+k_{2}^{2}u_{1}=0,& \quad \mbox{in} \quad \theta\in(\pi/2,3\pi/2), \vspace{0.2cm} \\
 u_{1}^{+}=u_{1}^{-}, \quad \partial_\nu^{+}u_{1}=\lambda\,\partial_\nu^{-}u_{1},  & \quad \mbox{on} \quad \theta=\pi/2,\,3\pi/2,
 \end{array}\right. \\
\label{a2}
&\left\{\begin{array}{lll}
 \Delta u_{2}+k_{1}^{2}u_{2}=0, & \quad \mbox{in} \quad \theta\in(0,\pi/2), \vspace{0.2cm} \\
 \Delta u_{2}+k_{2}^{2}u_{2}=0,& \quad \mbox{in} \quad \theta\in(\pi/2,2\pi), \vspace{0.2cm} \\
 u_{2}^{+}=u_{2}^{-}, \quad \partial_\nu^{+}u_{2}=\lambda\,\partial_\nu^{-}u_{2},  & \quad \mbox{on} \quad \theta=0,\,\pi/2.
 \end{array}\right.
\end{align}

By Proposition \ref{prop} (ii), the Cauchy data $(u_{1}^{+}, \partial_{\nu}u_{1}^{+}) $ are analytic on $B_{R}\cap\Lambda_{2}$. Then, the coincidence $u_1(r,\theta)=u_2(r,\theta)$ for all $\theta\in [0,\pi/2]$ implies that $u_{2}^{+}$ and $\partial_\nu u_{2}^{+}$ are both analytic on $B_{R}\cap\Lambda_{2}$. By the Cauchy-Kowalewski theorem in a piecewise analytic domain (refer to Lemma 2.1 in \cite{LHY}), we conclude that there exists $R_{1}\in(0,R)$ such that $u_{2}$ can be extended analytically from $B_{R_{1}}\cap\Omega_{\Lambda_{2}}^{+}$ to $B_{R_{1}}$ and the extended function $w_{2}$ satisfies that
\begin{equation*}
\left\{\begin{array}{lll}
\Delta w_{2}+k_{1}^{2}w_{2}=0, &\quad {\rm in}\quad B_{R_{1}}, \vspace{0.2cm} \\
 w_{2}=u_{2}^{+}, \quad \partial_\nu w_{2}=\partial_\nu u_{2}^{+},  & \quad \mbox{on} \quad B_{R_{1}}\cap\Lambda_{2}.
\end{array}\right.
\end{equation*}
Recalling the transmission boundary in (\ref{a2}) and the fact that $\lambda$ is a constant, we also find that $u_{2}^{-}$ and $\partial_\nu u_{2}^{-}$ are both analytic on $B_{R}\cap\Lambda_{2}$. Similarly, the solution $u_{2}$ can be extended analytically from $B_{R_{2}}\cap\Omega_{\Lambda_{2}}^{-}$ to $B_{R_{2}}$ ($R_{2}\in(0,R_{1})$) by the Cauchy-Kowalewski theorem. Denote by $v_{2}$ the extended function in $B_{R_{2}}$, which satisfies
\begin{equation*}
\left\{\begin{array}{lll}
\Delta v_{2}+k_{2}^{2}v_{2}=0, &\quad {\rm in}\quad B_{R_{2}}, \vspace{0.2cm} \\
 v_{2}=u_{2}^{-}, \quad \partial_\nu v_{2}=\partial_\nu u_{2}^{-},  & \quad \mbox{on} \quad B_{R_{2}}\cap\Lambda_{2}.
\end{array}\right.
\end{equation*}
Again using the transmission conditions in (\ref{a2}) yields
\begin{equation*}
\left\{\begin{array}{lll}
\Delta w_{2}+k_{1}^{2}w_{2}=0, &\quad {\rm in}\quad B_{R_{2}},\vspace{0.2cm} \\
\Delta v_{2}+k_{2}^{2}v_{2}=0, &\quad {\rm in}\quad B_{R_{2}}, \vspace{0.2cm} \\
w_{2}=v_{2}, \quad \partial_\nu w_{2}=\lambda\partial_\nu v_{2},  & \quad \mbox{on} \quad B_{R_{2}}\cap\Lambda_{2}.
\end{array}\right.
\end{equation*}
Since $k_{1}\neq k_{2}$, we obtain $w_{2}=v_{2}\equiv0$ in $B_{R_{2}}$ by Lemma \ref{Th1}, that is, $u_{2}\equiv0$ in $B_{R_{2}}$. This together with the unique continuation leads to $u_{2}\equiv0$ in $B_{R}$, which is impossible.

\section{Appendix}
This section is devoted to the regularity problem around a corner point and up to the flat interface,
and the well-posedness of solutions to the forward scattering (\ref{a})--(\ref{rad1}).

\subsection{Regularity around a corner}
Firstly, we investigate the regularity of a solution to the transmission problem of the Helmholtz equation in a right angle domain (see the Figure \ref{ff1}).
\begin{figure}[h] \label{ff1}
  \centering
  \begin{tikzpicture}
  \filldraw[gray,opacity=0.5] (0,0)--(0,-2.5) arc (270:360:2.5);
  \filldraw[gray!30,opacity=0.5] (0,0)--(2.5,0) arc (0:270:2.5);
  \draw[gray, thick,->] (0,-3) -- (0,3.5); \draw (0.1,3.7) node{$x_2$};
  \draw[gray, thick,->] (-3,0) -- (3.5,0); \draw (3.8,0) node{$x_1$};
  \filldraw[black] (0,0) circle (0.5pt); \draw (-0.2,0.2) node{$O$}; \filldraw[black] (0,0) circle (0.8pt);
  \draw[gray,thick] (0,0) circle (2.5cm); \draw (2.2,2.2) node{$B_{R}$};
  \draw (1,-1) node{$\Omega_{2}$}; \draw (-1,1) node{$\Omega_{1}$};
  \draw[red, line width=1] (0,0) -- (2.5,0); \draw (1.25,0.2) node{$\Pi_{1}$};
  \draw[red, line width=1] (0,0) -- (0,-2.5); \draw (-0.25,-1.25) node{$\Pi_{2}$};
  \end{tikzpicture}
  \caption{Sketch map of $\Omega_{\ell}$ and $\Pi_{\ell}$ ($\ell=1,2$).}
\end{figure}
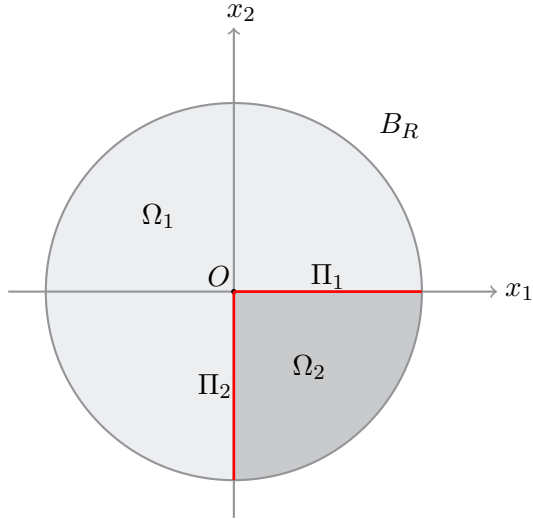
\begin{theorem} \label{Theo1}
The solution $\hat{u}$ to (\ref{ra}) has the regularity $\hat{u}\in H^{1+s}(B_{R})\cap H^{1+2/3}(\Omega_{\ell})$ for any $0\leq s<1/2$ ($\ell=1,2$).
\end{theorem}
\begin{proof}
For the sake of notational simplicity, we write $\varphi(\theta):=\varphi_{j}(\theta)$, $\eta:=\eta_{j}$ for some fixed $j$. A general solution to (\ref{b}) takes the form
\begin{equation} \label{func}
\varphi(\theta)=\left\{\begin{array}{l}
A^{+}\cos(\eta\theta)+B^{+}\sin(\eta\theta),\quad \theta\in(0,3\pi/2),  \vspace{0.1cm}\\
A^{-}\cos(\eta\theta)+B^{-}\sin(\eta\theta),\quad \theta\in(-\pi/2,0).
\end{array}\right.
\end{equation}
Using the transmission boundary conditions in (\ref{b}) yields
\begin{equation*}
A^{+}=A^{-}, \quad
A^{+}\cos(3\pi\eta/2)+B^{+}\sin(3\pi\eta/2)=A^{-}\cos(\pi\eta/2)-B^{-}\sin(\pi\eta/2).
\end{equation*}
Since
\begin{equation*}
\varphi^{'}(\theta)=\left\{\begin{array}{l}
-\eta A^{+}\sin(\eta\theta)+\eta B^{+}\cos(\eta\theta),\quad \theta\in(0,3\pi/2), \vspace{0.1cm} \\
-\eta A^{-}\sin(\eta\theta)+\eta B^{-}\cos(\eta\theta),\quad \theta\in(-\pi/2,0),
\end{array}\right.
\end{equation*}
 we have
\begin{equation*}
B^{+}=\lambda B^{-},\quad
-A^{+}\sin(3\pi\eta/2)+B^{+}\cos(3\pi\eta/2)=\lambda\,\big[A^{-}\sin(\pi\eta/2)+B^{-}\cos(\pi\eta/2)\big].
\end{equation*}
That is, $(A^{+},A^{-},B^{+},B^{-})$ satisfies the following 4-by-4 algebraic system:
\begin{equation*}
\left(\begin{array}{cccc}
1 & -1 & 0 & 0 \\
\cos(3\pi\eta/2) & -\cos(\pi\eta/2) & \sin(3\pi\eta/2) & \sin(\pi\eta/2) \\
0 & 0 & 1 & -\lambda \\
\sin(3\pi\eta/2) & \lambda\sin(\pi\eta/2) & -\cos(3\pi\eta/2) & \lambda\cos(\pi\eta/2)
\end{array}\right)\left(\begin{array}{cccc} A^{+}\\ A^{-}\\ B^{+}\\ B^{-}\end{array}\right) =\left(\begin{array}{cccc} 0\\ 0\\ 0\\ 0\end{array}\right).
\end{equation*}
We denote the fourth order matrix on the left by $M$. Then simple calculation shows that
\begin{align*}
|M|=&\left|\begin{array}{cccc}
1 & -1 & 0 & 0 \\
\cos(3\pi\eta/2) & -\cos(\pi\eta/2) & \sin(3\pi\eta/2) & \sin(\pi\eta/2) \\
0 & 0 & 1 & -\lambda \\
\sin(3\pi\eta/2) & \lambda\sin(\pi\eta/2) & -\cos(3\pi\eta/2) & \lambda\cos(\pi\eta/2)
\end{array}\right|   \\
=&\left|\begin{array}{ccc}
\cos(3\pi\eta/2)-\cos(\pi\eta/2) & \sin(3\pi\eta/2) & \sin(\pi\eta/2) \\
0 & 1  & -\lambda \\
\lambda\sin(\pi\eta/2)+\sin(3\pi\eta/2) & -\cos(3\pi\eta/2) & \lambda\cos(\pi\eta/2)
\end{array}\right|  \\
=&\left|\begin{array}{ccc}
\cos(3\pi\eta/2)-\cos(\pi\eta/2) & 0 & \sin(\pi\eta/2)+\lambda\sin(3\pi\eta/2) \\
0 & 1  & -\lambda \\
\lambda\sin(\pi\eta/2)+\sin(3\pi\eta/2) & 0 & \lambda\cos(\pi\eta/2)-\lambda\cos(3\pi\eta/2)
\end{array}\right|  \\
=&\left|\begin{array}{ccc}
\cos(3\pi\eta/2)-\cos(\pi\eta/2) & \sin(\pi\eta/2)+\lambda\sin(3\pi\eta/2) \\
\lambda\sin(\pi\eta/2)+\sin(3\pi\eta/2)  & \lambda\cos(\pi\eta/2)-\lambda\cos(3\pi\eta/2)
\end{array}\right|.
\end{align*}
That is,
\begin{align*} |M|&=-\lambda\big[\cos(3\pi\eta/2)-\cos(\pi\eta/2)\big]^{2}-\big[\lambda\sin(\pi\eta/2)+\sin(3\pi\eta/2)\big]\big[\sin(\pi\eta/2)+\lambda\sin(3\pi\eta/2)\big]\\
&=2\lambda\cos(3\pi\eta/2)\cos(\pi\eta/2)-(\lambda^{2}+1)\sin(3\pi\eta/2)\sin(\pi\eta/2)-2\lambda \\
&=(\lambda+1)^{2}\cos^{2}(\pi\eta)-\frac{(\lambda-1)^{2}}{2}\cos(\pi\eta)-\frac{\lambda^{2}+6\lambda+1}{2}=0,
\end{align*}
which implies that
\begin{align*}
\cos(\pi\eta)=-\frac{\lambda^{2}+6\lambda+1}{2(\lambda+1)^{2}} \quad {\rm or} \quad \cos(\pi\eta)=1.
\end{align*}
Hence,
\begin{equation*}
\eta =\frac{1}{\pi}\arccos\Big(-\frac{\lambda^{2}+6\lambda+1}{2(\lambda+1)^{2}}\Big) \quad {\rm or}\quad \eta=2l, \quad l\in\mathbb{N}.
\end{equation*}
Note that, $\eta\in(0,1)$ and
\begin{equation*}
\frac{\lambda^{2}+6\lambda+1}{2(\lambda+1)^{2}} =\frac{(\lambda+1)^{2}+4\lambda}{2(\lambda+1)^{2}} =\frac{1}{2}+\frac{2\lambda}{(\lambda+1)^{2}}\in (1/2,1), \quad {\rm i.e.}~ -1<\cos(\pi\eta)<-\frac{1}{2}.
\end{equation*}
Therefore,
\begin{equation*}
\eta=\frac{1}{\pi} \arccos\Big(-\frac{\lambda^{2}+6\lambda+1}{2(\lambda+1)^{2}}\Big)>\frac{2}{3}.
\end{equation*}
The proof is complete.
\end{proof}
\subsection{Regularity up the flat interface}
In this subsection we suppose that the angle is $\pi$ and consider the transmission problem
\begin{equation} \label{rb}
\left\{\begin{array}{lll}
\Delta v_{\ell}+k_{\ell}^{2}v_{\ell}=0,&\quad \mbox{in}\quad\widetilde{\Omega}_{\ell}, \\
v_{1}=v_{2}, \quad \partial_{\nu}v_{1}=\lambda\partial_{\nu}v_{2},&\quad \mbox{on}\quad\widetilde{\Pi}_{\ell},
\end{array}\right.
\end{equation}
where $k_{\ell}$ are constants and $k_{1}\neq k_{2}$, the unit normal vector $\nu$ at $\widetilde{\Pi}_{\ell}$ is pointing into $\widetilde{\Omega}_{1}$. The two semi-circles $\widetilde{\Omega}_{\ell}$ and their boundaries $\widetilde{\Pi}_{\ell}$ ($\ell=1,2$) are defined as (see the Figure \ref{ff2}):
\ben
&&\widetilde{\Omega}_{1}:=\{(r,\theta):0<r<R,~0\leq\theta<\pi/2~{\rm or}~3\pi/2<\theta\leq2\pi\}, \quad \widetilde{\Pi}_{1}:=\{(r,\pi/2):0\leq r\leq R\},\\
&&\widetilde{\Omega}_{2}:=\{(r,\theta):0<r<R,~\pi/2<\theta<3\pi/2\},  \quad \quad\quad\quad\quad\quad\quad~\widetilde{\Pi}_{2}:=\{(r,3\pi/2):0\leq r\leq R\}.
\enn
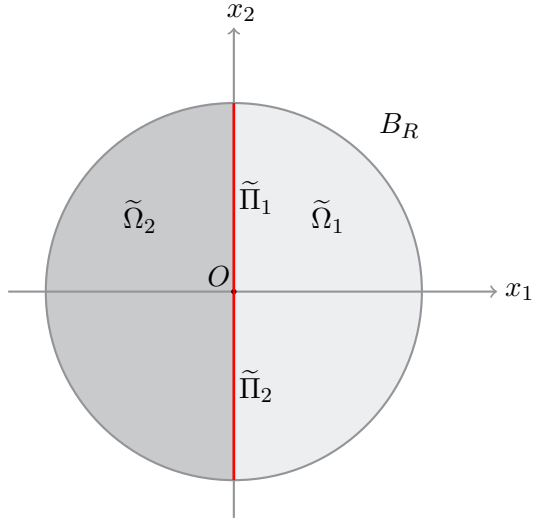
\begin{figure}[h]  \label{ff2}
  \centering
  \begin{tikzpicture}
  \filldraw[gray!30,opacity=0.5] (0,0)--(0,-2.5) arc (270:450:2.5);
  \filldraw[gray,opacity=0.5] (0,0)--(0,2.5) arc (90:270:2.5);
  \draw[gray, thick,->] (0,-3) -- (0,3.5); \draw (0.1,3.7) node{$x_2$};
  \draw[gray, thick,->] (-3,0) -- (3.5,0); \draw (3.8,0) node{$x_1$};
  \filldraw[black] (0,0) circle (0.5pt); \draw (-0.2,0.2) node{$O$};\filldraw[black] (0,0) circle (0.8pt);
  \draw[gray,thick] (0,0) circle (2.5cm); \draw (2.2,2.2) node{$B_{R}$};
  \draw (-1.25,1) node{$\widetilde{\Omega}_{2}$}; \draw (1.25,1) node{$\widetilde{\Omega}_{1}$};
  \draw[red, line width=1] (0,0) -- (0,2.5); \draw (0.3,1.25) node{$\widetilde{\Pi}_{1}$};
  \draw[red, line width=1] (0,0) -- (0,-2.5); \draw (0.3,-1.25) node{$\widetilde{\Pi}_{2}$};
  \end{tikzpicture}
  \caption{Sketch map of $\widetilde{\Omega}_{\ell}$ and $\widetilde{\Pi}_{\ell}$ ($\ell=1,2$).}
\end{figure}

In order to rewrite the equation (\ref{rb}) into the divergence form, we define
\begin{equation*}
\tilde{a}(\theta):=\left\{\begin{array}{ll}
1,  \quad {\rm in}~~\widetilde{\Omega}_{1}, \\
\lambda, \quad {\rm in}~~\widetilde{\Omega}_{2},
\end{array}\right. \quad\quad \tilde{\kappa}(\theta):=\left\{\begin{array}{ll}
k_{1}^{2},\quad & {\rm in}~~\widetilde{\Omega}_{1}, \\
\lambda k_{2}^{2},  \quad & {\rm in}~~\widetilde{\Omega}_{2},
\end{array}\right. \quad\quad \tilde{v}(r,\theta):=\left\{\begin{array}{ll}
v_{1},\quad {\rm in}~~\widetilde{\Omega}_{1}, \\
v_{2},\quad {\rm in}~~\widetilde{\Omega}_{2}.
\end{array}\right.
\end{equation*}
Then (\ref{rb}) is equivalent to
\begin{equation*}
\nabla\cdot(\tilde{a}(\theta)\nabla \tilde{v})+\tilde{\kappa}(\theta)\tilde{v}=0 \quad \mbox{in}~B_{R}.
\end{equation*}
By the decomposition theorem,  $\tilde{v}=\tilde{w}+\sum\limits_{j=1}^{m}\tilde{c}_{j}r^{\delta_{j}}\phi_{j}(\theta)(\ln r)^{\tilde{p}_{j}}$ in $B_{R}$ with $\tilde{p}_{j}\in\{0,1,\cdots\}$. Here, $\tilde{w}\in H^{2}(\widetilde{\Omega}_{\ell})$ ($\ell=1,2$), and $\delta_{j}\in(0,1)$ are eigenvalues of the following positive definite Sturm-Liouville:
\begin{equation} \label{b2}
\left\{\begin{array}{lll}
\phi_{j}^{''}(\theta)+\delta_{j}^{2}\phi_{j}(\theta)=0,\quad & \mbox{in} \quad \theta\in[0,\pi/2)\cup(\pi/2,3\pi/2)\cup(3\pi/2,2\pi],\vspace{0.2cm} \\
\phi_{j,+}(\pi/2)=\phi_{j,-}(\pi/2), \quad & \phi^{'}_{j,+}(\pi/2)=\lambda \phi_{j,-}^{'}(\pi/2),\vspace{0.2cm}\\
\phi_{j,+}(3\pi/2)=\phi_{j,-}(3\pi/2),\quad & \phi_{j,+}^{'}(3\pi/2)=\lambda \phi_{j,-}^{'}(3\pi/2).
\end{array}\right.
\end{equation}
Here, $\phi_{j,+}$, $\phi_{j,+}^{'}$ denote the limits from $\widetilde{\Omega}_{1}$ and $\phi_{j,-}$, $\phi_{j,-}^{'}$ the limits from $\widetilde{\Omega}_{2}$.

\begin{theorem} \label{Theo2}
The solution $\tilde{v}$ to (\ref{rb}) has the regularity $\tilde{v}\in H^{1+s}(B_{R})\cap H^{2}(\widetilde{\Omega}_{\ell})$ for any $0\leq s<1/2$, and $\tilde{v}$ is analytic on the closure of $\widetilde{\Omega}_{\ell}$ ($\ell=1,2$).
\end{theorem}
\begin{proof}
Write $\phi(\theta):=\phi_{j}(\theta)$, $\delta_{j}:=\delta$ for some fixed $j$. A general solution to (\ref{b2}) takes the form
\begin{equation*}
\phi(\theta)=\left\{\begin{array}{l}
\tilde{A}^{+}\cos(\delta\theta)+\tilde{B}^{+}\sin(\delta\theta),\quad \theta\in[0,\pi/2)\cup(3\pi/2,2\pi], \vspace{0.1cm}\\
\tilde{A}^{-}\cos(\delta\theta)+\tilde{B}^{-}\sin(\delta\theta),\quad \theta\in(\pi/2,3\pi/2).
\end{array}\right.
\end{equation*}
Using the transmission boundary conditions in (\ref{b2}) yields
\begin{equation*}
\left\{\begin{array}{l}
\tilde{A}^{+}\cos(\pi\delta/2)+\tilde{B}^{+}\sin(\pi\delta/2)=\tilde{A}^{-}\cos(\pi\delta/2)+\tilde{B}^{-}\sin(\pi\delta/2), \vspace{0.1cm}\\
\tilde{A}^{+}\cos(3\pi\delta/2)+\tilde{B}^{+}\sin(3\pi\delta/2)=\tilde{A}^{-}\cos(3\pi\delta/2)+\tilde{B}^{-}\sin(3\pi\delta/2).
\end{array}\right.
\end{equation*}
Since
\begin{equation*}
\phi^{'}(\theta)=\left\{\begin{array}{l}
-\delta \tilde{A}^{+}\sin(\delta\theta)+\delta \tilde{B}^{+}\cos(\delta\theta),\quad \theta\in[0,\pi/2)\cup(3\pi/2,2\pi], \vspace{0.1cm}\\
-\delta \tilde{A}^{-}\sin(\delta\theta)+\delta \tilde{B}^{-}\cos(\delta\theta),\quad \theta\in(\pi/2,3\pi/2),
\end{array}\right.
\end{equation*}
then we obtain that
\begin{equation*}
\left\{\begin{array}{l}
-\tilde{A}^{+}\sin(\pi\delta/2)+\tilde{B}^{+}\cos(\pi\delta/2)=\lambda[-\tilde{A}^{-}\sin(\pi\delta/2)+\tilde{B}^{-}\cos(\pi\delta/2)],\vspace{0.1cm}\\
-\tilde{A}^{+}\sin(3\pi\delta/2)+\tilde{B}^{+}\cos(3\pi\delta/2)=\lambda[-\tilde{A}^{-}\sin(3\pi\delta/2)+\tilde{B}^{-}\cos(3\pi\delta/2)].
\end{array}\right.
\end{equation*}
That is, $(\tilde{A}^{-},\tilde{B}^{-},\tilde{A}^{+},\tilde{B}^{+})$ satisfies the following equation system:
\begin{equation*}
\left(\begin{array}{cccc}
\cos(\pi\delta/2) & \sin(\pi\delta/2) & -\cos(\pi\delta/2) & -\sin(\pi\delta/2) \\
\cos(3\pi\delta/2) & \sin(3\pi\delta/2) & -\cos(3\pi\delta/2) & -\sin(3\pi\delta/2) \\
-\lambda\sin(\pi\delta/2) & \lambda\cos(\pi\delta/2)  &  \sin(\pi\delta/2) &  -\cos(\pi\delta/2) \\
-\lambda\sin(3\pi\delta/2) & \lambda\cos(3\pi\delta/2)  & \sin(3\pi\delta/2) & -\cos(3\pi\delta/2)
\end{array}\right)\left(\begin{array}{cccc}\tilde{A}^{-} \\ \tilde{B}^{-} \\ \tilde{A}^{+}\\ \tilde{B}^{+}\end{array}\right)
=\left(\begin{array}{cccc} 0\\ 0\\ 0\\ 0\end{array}\right).
\end{equation*}
We denote the fourth order matrix on the left by $\tilde{M}$. Then simple calculation shows that
\begin{align*}
|\tilde{M}|=&\left|\begin{array}{cccc}
\cos(\pi\delta/2) & \sin(\pi\delta/2)  & 0 & 0 \\
\cos(3\pi\delta/2) & \sin(3\pi\delta/2) & 0 & 0 \\
-\lambda\sin(\pi\delta/2) & \lambda\cos(\pi\delta/2) & (1-\lambda)\sin(\pi\delta/2) & (\lambda-1)\cos(\pi\delta/2) \\
-\lambda\sin(3\pi\delta/2) & \lambda\cos(3\pi\delta/2) & (1-\lambda)\sin(3\pi\delta/2) & (\lambda-1)\cos(3\pi\delta/2)
\end{array}\right|  \\
=&-(\lambda-1)^{2}\left|\begin{array}{cccc}
\cos(\pi\delta/2) & \sin(\pi\delta/2)  & 0 & 0 \\
\cos(3\pi\delta/2) & \sin(3\pi\delta/2) & 0 & 0 \\
0 & 0 & \sin(\pi\delta/2) & \cos(\pi\delta/2) \\
0 & 0 & \sin(3\pi\delta/2) & \cos(3\pi\delta/2)
\end{array}\right| \\
=&\,(\lambda-1)^{2}\sin^{2}(\pi\delta)=0.
\end{align*}
That is, $\sin(\pi\delta)=0$ and then $\delta\in\mathbb{N}$, which implies that $\tilde{v}\in H^{1}(B_{R})\cap H^{2}(\widetilde{\Omega}_{\ell})$ and $\tilde{v}$ is analytic up to the boundary of $\widetilde{\Pi}_{1}\cup\widetilde{\Pi}_{2}$. The proof is complete.
\end{proof}

\subsection{Uniqueness and existence of forward scattering problem} Define the DtN mapping $T: H_\alpha^{1/2}(\Gamma_H)\rightarrow H_\alpha^{-1/2}(\Gamma_H)$ by
\ben
(T f)(x_1):=\sum_{n\in \mathbb{Z}} i\,\beta_n f_n\,e^{i\alpha_n x_1},\qquad {\rm where}~~ f(x_1)=\sum_{n\in \mathbb{Z}} f_n\,e^{i\alpha_n x_1}\in H_\alpha^{1/2}(\Gamma_H).
\enn
Introduce the piecewise analytic functions
\ben
a(x):=\left\{\begin{array}{lll}
1&\mbox{in}\quad S_{H}^+,\\
\lambda & \mbox{in} \quad S_{H}^-,
\end{array}\right.\quad\quad
\kappa(x):=\left\{\begin{array}{lll}
k_{1}^2 & \mbox{in}\quad S_{H}^+,\\
\lambda\,k_{2}^2 &\mbox{in} \quad S_{H}^-.
\end{array}\right.
\enn
The scattering problem \eqref{a}--\eqref{rad1} can be equivalently formulated as the following divergence form in the truncated domain $S_H$:
\begin{equation} \label{model}
\left\{\begin{array}{lll}
\nabla\cdot(a(x)\nabla u)+\kappa(x)u=0,& \quad \mbox{in}\quad S_H, \vspace{0.1cm} \\ \partial_2 u=Tu+(\partial_2 u^i-Tu^i),&\quad\mbox{on}\quad \Gamma_H, \\
u=0, &\quad\mbox{on}\quad \Gamma_0.
\end{array}\right.
\end{equation}

\begin{theorem} \label{TH}
The boundary value problem \eqref{model} has at least one solution $u\in H^1_\alpha(S_H)$ for any fixed $H>\Lambda^+$. Moreover, uniqueness remains true for any $k_{1},k_{2}>0$ under the following monotonicity conditions on the medium:
\be \label{q}
k_{1}^{2}>\lambda k_{2}^{2}.
\en
\end{theorem}
\begin{proof}
From the definition of $T$, it follows that for $f\in H^{1/2}_\alpha(\Gamma_H)$,
\be  \label{re}
\mbox{Re}\,\langle T\,f, f\rangle=-\sum_{ |\alpha_n|>k_{1}} |\beta_n|\, |f_n|^2\leq 0,\quad \mbox{Im}\,\langle T\,f, f\rangle=\sum_{ |\alpha_n|\leq k_{1}} |\beta_n\, |f_n|^2\geq 0,
\en
where the pair $\langle \cdot, \cdot\rangle$ denotes the duality between $H_\alpha^{-1/2}$ and $H_\alpha^{1/2}$ on $\Gamma_{H}$. The variational formulation for \eqref{model} can be written as: find $u\in H_\alpha^1(S_H)$ such that for all $v\in H_\alpha^1(S_H)$,
\be  \label{va}
\quad\quad L(u, v):=\int_{S_H} \left[a(x) \nabla u\cdot \nabla \overline{v}- a(x)\kappa(x) u\overline{v}\right]\,\mbox{d}x-\int_{\Gamma_H} Tu \overline{v}\,\mbox{d}s=\int_{\Gamma_H}\left(T u^i-\frac{\partial u^i}{\partial x_2}\right)\overline{v}\,\mbox{d}s.
\en
Using $\eqref{re}$, one can conclude that the above sesquilinear form gives rise to a strongly elliptic operator $\mathcal{L}$ such that $L(u, v)=\langle \mathcal{L}u, v\rangle$ for all  $u, v\in H_\alpha^{1/2}(S_H)$ (see also e.g., \cite{D93, ES98}), where $\langle\cdot, \cdot\rangle$ denotes the inner product over the Hilbert space $H^1_\alpha(S_H)$.  On the other hand, the adjoint of $\mathcal{L}$: $H^1_\alpha(S_H)\rightarrow H^1_\alpha(S_H)$ takes the explicit form
\ben
\langle \mathcal{L}^*u, v\rangle=\overline{L(v, u)}=\int_{S_H} \left[a(x) \nabla u\cdot \nabla \overline{v}- a(x) \kappa(x) u\overline{v}\right]\,\mbox{d}x+2\pi\sum_{n\in \Z} i\overline{\beta_n} u_n \overline{v}_n,\quad u, v\in H^1_\alpha(S_H).
\enn
Here, $u_n$ and $v_n$ denote the Fourier coefficients of $e^{-i\alpha x_1}u|_{\Gamma_H}$ and $e^{-i\alpha x_1}v|_{\Gamma_H}$, respectively. Taking the imaginary part on both sides of the previous identity with $v=u$ and using \eqref{re}, we get $\sum\limits_{|\alpha_n|\leq k_{1}} |\beta_n|\, |u_n|^2=0$ for $u\in {\rm Ker}(L^*)$. This implies that
\ben
\int_{\Gamma_H}\left(Tu^i-\frac{\partial u^i}{\partial x_2}\right)\overline{v}\,\mbox{d}s=0\quad\mbox{for all}\quad v\in {\rm Ker}(\mathcal{L}^*).
\enn
By Fredholm alternative, there always exists a solution $u\in H^1_\alpha(S_H)$ to \eqref{model}.

To prove uniqueness, we suppose that $u^i\equiv 0$. Then $u$ satisfies the upward Rayleigh expansion radiation condition. Taking the real part on both sides of \eqref{va} with $v=u$ and $u^i=0$ and using \eqref{re}, we obtain
\ben
I_1:=\int_{S_H} \left[a(x) |\nabla u|^2- a(x) \kappa (x) |u|^2\right]\,\mbox{d}x
= -\sum_{|\alpha_n|> k_{1}} |\beta_n|\, |u_n|^2\, e^{-2|\beta_n|\,H} \leq 0.
\enn
Multiplying the Helmholtz equation by $x_2\,\partial_2\overline{u}$ and integrating by part over $S_{H}^\pm$ yield the Rellich's identities:
\ben
I^+&=&\left(\int_{\Gamma_H}-\int_\Lambda \right) x_2\left[
-\nu_2 |\nabla u |^2+\nu_2 k_{1}^2 |u|^2+2\mbox{Re}(\partial_2 \overline{u^+}\,\partial_\nu u^+)
\right]\,\mbox{d}s \\ \nonumber
&&+\int_{S_{H}^+} |\nabla u|^2-k_{1}^2\,|u|^2-2|\partial_2 u|^2\, \mbox{d}x=0,\\
I^-&=&\int_\Lambda x_2\left[
-\nu_2 |\nabla u|^2+\nu_2 k_{2}^2|u|^2+2\mbox{Re}(\partial_2 \overline{u^{-}}\,\partial_\nu u^-)\right]\,\mbox{d}s\\
&&+\int_{S_{H}^-} |\nabla u|^2-k_{1}^2\,|u|^2-2|\partial_2 u|^2\, \mbox{d}x
=0.   
\enn
The integrand over $\Lambda$ is well-defined because, for rectangular gratings it holds that $u\in H^{3/2+\epsilon}_\alpha(S_{H}^\pm)$ for some $\epsilon>0$ depending on $\lambda$ (see e.g., \cite[Chapter 2.4.3]{Pe2001} and \cite[Section 3.3]{ES98}). Straightforward calculations show that
\ben
\int_{\Gamma_H} x_2\left[
-\nu_2 |\nabla u |^2+\nu_2 k_{1}^2 |u|^2+2\mbox{Re}(\partial_2 \overline{u}\,\partial_\nu u)
\right]\,\mbox{d}s=H\sum_{ |\alpha_n|\leq k_{1}} |\beta_n|\, |u_n|^2
= 0,
\enn
and
\be\nonumber
0&=&I^++\lambda\, I^- \\ \nonumber
&=&-\int_\Lambda \left[ \lambda (\lambda-1)|\partial_\nu u^-|^2+(\lambda-1)|\partial_\tau u^-|^2+(k_{1}^2-\lambda k_{2}^{2})|u|^2 \right] \nu_2 x_2\,\mbox{d}s-2\int_{S_H}a(x) |\partial_2 u|^2\,\mbox{d}x+ I_1, \label{identity} 
\en    
where $\partial_\tau$ denotes the tangential derivative on $\Lambda$ with $\tau:=(-\nu_2, \nu_1)$. By the assumptions \eqref{q} on $k_{1},k_{2}$ and recalling the fact that $\nu_2\geq 0$ on $\Lambda$, we conclude that the integral over $\Lambda$ is non-positive, so that each term in the above expression vanishes. Consequently, we get $\partial_2 u\equiv 0$ in $S_H$ and $I_1=0$, implying that $u_{n}=0$ for all $|\alpha_n|>k_{1}$. Therefore,
\ben
u=A_n e^{ik_{1}x_1}+ A_m\,e^{-ik_{1}x_1}\quad\mbox{in}\quad \Omega_\Lambda^+,\qquad A_n, A_m\in \mathbb{C},
\enn
if $\alpha_n=k_{1}$ or $\alpha_m=-k_{1}$ for some $n,m\in \mathbb{Z}$ (that is, Rayleigh frequencies occurs). Note that the above expression of $u$ is well-defined in $\mathbb{R}^2$. Since $\nu_2=1$ on the line segment of $\Lambda$ parallel to the $x_1$-axis and $k_{1}^{2}>\lambda k_{2}^{2}$, one can also deduce from \eqref{identity} that $u \equiv 0$ on this segment, which gives $A_n=A_m=0$ and thus $u\equiv 0$.
\end{proof}

\vspace{0.3cm}
{\bf Acknowledgments.}

The work of G. Hu is partially supported by the National Natural Science Foundation of China (No. 12071236) and the Fundamental Research Funds for Central Universities in China (No. 63213025). The work of J. Xiang is supported by the Natural Science Foundation of Hubei (No. 2022CFB725).


\begin{thebibliography}{<num>}



\bibitem{Bao2001}G. Bao, L. Cowsar and W. Masters. Mathematical Modeling in Optical Science. Philadelphia, USA: SIAM, 2001.

\bibitem{BL}G. Bao and P. Li. Maxwell's Equations in Periodic Structures, Springer, Singapore, 2022.


\bibitem{BFI} H. Bellout, A. Friedman and V. Isakov. Stability for an inverse problem in potential theory, Trans. Amer. Math. Soc., 332 (1992): 271-296.

\bibitem{BS} A. S. Bonnet-Bendhia and F. Starling. Guided waves by electromagnetic gratings and non-uniqueness examples for the diffraction problem. Math. Methods Appl. Sci., 17 (1994): 305-338.




\bibitem{D93} D. C. Dobson. Optimal design of periodic antireflective structures for the Helmholtz equation, European J. Appl. Math., 4 (1993): 321-340.


\bibitem{Elschner2018} J. Elschner and G. Hu. Acoustic scattering from corners, edges and circular cones, Archive for Rational Mechanics and Analysis, 228 (2018): 653-690.

\bibitem{ElschnerJ2015} J. Elschner and G. Hu. Corners and edges always scatter, Inverse Probl., 31 (2015): 015003.


\bibitem{EHY} J. Elschner, G. Hu and M. Yamamoto. Uniqueness in inverse elastic scattering from unbounded rigid surfaces of rectangular type, Inverse Problems and Imaging 9 (2015): 127-141.



\bibitem{ES98} J. Elschner and G. Schmidt. Diffraction in periodic structures and optimal design of binary gratings. I. Direct problems and gradient formulas, Math. Methods Appl. Sci., 21 (1998): 1297-1342.

\bibitem{Hettlich1997} F. Hettlich and A. Kirsch. Schiffer's theorem in inverse scattering for periodic structures, Inverse Probl., 13 (1997): 351-361.



\bibitem{XH} J.Xiang and G.Hu. Uniqueness in determining rectangular grating proles with a single incoming wave (Part I): TE polarization case, Inverse Problem 39 (2023): 055004.


\bibitem{HK22} G. Hu and A. Kirsch. Direct and inverse time-harmonic scattering by Dirichlet periodic curves with local perturbations, to appear.


\bibitem{Kirsch1994} A. Kirsch. Uniqueness theorems in inverse scattering theory for periodic structures, Inverse Probl., 10 (1994): 145-152.





\bibitem{K1967} V. A. Kondratiev. Boundary value problems for elliptic equations in domains with conical or angular points, Trans. Moscow Math. Soc., 16 (1967): 227-313.

\bibitem{KMR} V. A. Kozlov, V. G. Maz'ya and J. Rossmann. Elliptic Boundary Value Problems in Domains with Point Singularities, American Mathematical Society, Providence, RI, 1997.


\bibitem{LHY} L. Li, G. Hu and J. Yang. Piecewise-analytic interfaces with weakly singular points of arbitrary order always scatter, Journal of Functional Analysis, 284 (2023): 109800.

\bibitem{Ray1907} J. W. S. Lord Rayleigh. On the dynamical theory of gratings, Proc. Roy. Soc. Lond. A, 79 (1907): 399-416.

\bibitem{MNP} V. G. Maz'ya, S. A. Nazarov and B. A. Plamenevskii. Asymptotic Theory of Elliptic Boundary Value Problems in Singularly Perturbed Domains I, Birkh auser-Verlag, Basel, 2000.



\bibitem{Petit1980} R. Petit. Electromagnetic Theory of Gratings (Topics in Current Physics vol 22), (Heidelberg: Springer), 1980.

\bibitem{Pe2001} M. Petzoldt. Regularity and error estimators for elliptic problems with discontinuous coefficients, PhD Thesis, Berlin: Free University, 2001. Available online at:  http://www.diss.fu-berlin.de/diss

\bibitem{Pe2005} Petzoldt M. Regularity results for interface problems in 2D, WIAS Preprint No. 565 (2000), DOI: 10.20347/WIAS.PREPRINT.565

\bibitem{SK97} B. Schnabel and E. B. Kley. Fabrication and application of subwavelength gratings, Proc. SPIE, 3008 (1997): 233-241.




\bibitem{Turunen1997} J. Turunen and F. Wyrowski. Diffractive Optics for Industrial and Commercial Applications, Berlin: Akademie, 1997.



\bibitem{XH23}X. Xu, G. Hu, B. Zhang and H. Zhang. Uniqueness in inverse diffraction grating problems with infinitely many plane waves at a fixed frequency, SIAM J. Appl. Math 83 (2023): 302-326.
\end{thebibliography}
\end{document}